\documentclass[12pt, reqno]{amsart}
\usepackage{amsfonts}
\usepackage{bbm}
\usepackage{amscd,amsfonts}
\usepackage{amssymb, eucal, amsfonts, amsmath, xypic, latexsym, tikz}
\usepackage{pifont}
\usepackage{mathrsfs,color}
\usepackage{amsthm,indentfirst,bm,fancyhdr,dsfont}
\usepackage{graphicx}
\usepackage[all]{xy}
\usepackage[CJKbookmarks=true]{hyperref}

\usepackage{mathrsfs}
\usepackage{amsmath}
\usepackage{amssymb}
\usepackage{hyperref}

\newtheorem{theorem}{Theorem}[section]
\newtheorem{lemma}[theorem]{Lemma}

\newtheorem{proposition}[theorem]{Proposition}
\newtheorem{corollary}[theorem]{Corollary}
\theoremstyle{definition}

\newtheorem{definition}[theorem]{Definition}

\newtheorem{example}[theorem]{Example}

\newtheorem{remark}[theorem]{Remark}
\newtheorem{conjecture}[theorem]{Conjecture}
\newtheorem{question}[theorem]{Question}

\numberwithin{equation}{section}

\def\ggg{\mathfrak{g}}

\def\gl{\mathfrak{gl}}

\def\ggg{\mathfrak{g}}
\def\gd{\dot{g}}

\def\uuu{\mathfrak{u}}

\def\caln{\mathcal{N}}
\def\cala{\mathcal{A}}
\def\cals{\mathcal{S}}
\def\calc{\mathcal{C}}
\def\cald{\mathcal{D}}
\def\cale{\mathcal{E}}
\def\calf{\mathcal{F}}
\def\calh{\mathcal{H}}

\def\calp{\mathscr{P}}
\def\scrl{\mathscr{L}}

\def\frakN{\mathfrak{N}}

\def\bc{\mathbf{c}}
\def\bi{\mathbf{i}}
\def\bj{\mathbf{j}}
\def\bk{\mathbf{k}}
\def\bl{\mathbf{l}}
\def\bh{\mathbf{h}}
\def\bp{\mathbf{p}}
\def\bq{\mathbf{q}}
\def\bs{\mathbf{s}}

\def\bt{{\mathbf{t}}}
\def\bx{{\mathbf{x}}}

\def\Gm{\mathbf{G_m}}

\def\bbc{\mathbb{C}}
\def\bbf{\mathbb{F}}

\def\bbn{\mathbb{N}}

\def\ug{\underline{g}}
\def\uG{\underline{G}}

\def\uu{\underline{u}}
\def\uGz{\underline{G_0}}

\def\um{\underline{m}}
\def\ur{\underline{r}}
\def\un{\underline{n}}
\def\ul{\underline{l}}

\def\uV{\underline{V}}

\def\vtr{V^{\otimes r}}
\def\uvtr{\underline{V}^{\otimes r}}
\def\duvtr{{\underline{V}^*}^{\otimes r}}
\def\fsr{\mathfrak{S}_r}
\def\fsl{\mathfrak{S}_l}
\def\fsm{\mathfrak{S}_m}
\def\fsk{\mathfrak{S}_k}
\def\fsru{\mathfrak{S}^{(r)}}
\def\ann{{\text{ann}}}

\def\res{\mathsf{Res}}

\def\bk{\mathbf{k}}

\def\tpis{\tilde\pi_\bs}

\def\sfC{\textsf{C}}

\def\Ann{\mbox{\rm Ann}}

\def\Lie{\mathsf{Lie}}

\def\GL{\text{GL}}

\def\End{\text{End}}

\def\Diag{\text{Diag}}

\def\id{\mathsf{id}}

{\vskip-\lastskip\medskip
  \noindent
  }%
{\qed\par\medskip
  }
\newcommand{\fz}{\mathbb{Z}} 

\newcommand{\cw}{\mathscr{W}}

\def\GL{\text{\rm GL}}
\def\diag{\text{\rm Diag}}

\def\td{\texttt{d}}

\begin{document}



\thanks{This work is partially supported by the National Natural Science Foundation of China (12071136, 11671138 and 11771279), Shanghai Key Laboratory of PMMP (No. 13dz2260400).}

\title[On enhanced reductive groups (I)]{On enhanced reductive groups (I): \\ Parabolic Schur algebras and the dualities related to degenerate double Hecke algebras}
\author{Bin Shu, Yunpeng Xue and Yufeng Yao}

\begin{center}\textsc{Dedicated to the memory of Professor Guang-Yu Shen}
\end{center}

\address{School of Mathematical Sciences, East China Normal University, Shanghai, 200241, China.} \email{bshu@math.ecnu.edu.cn}
\address{School of Mathematical Sciences, East China Normal University, Shanghai, 200241, China.} \email{1647227538@qq.com}
\address{Department of Mathematics, Shanghai Maritime University, Shanghai, 201306, China.}\email{yfyao@shmtu.edu.cn}

\begin{abstract} An enhanced algebraic group $\uG$ of $G=\GL(V)$ over $\bbc$ is a product variety $\GL(V)\times V$, endowed with an enhanced cross product. Associated with a natural tensor representation of $\uG$, there are naturally Levi and parabolic Schur algebras $\mathcal{L}$ and $\mathcal{P}$ respectively. We precisely investigate their structures, and study the dualities on the enhanced tensor representations for variant groups and algebras. In this course, an algebraic model  of  so-called degenerate double Hecke algebras (DDHA) is produced, and becomes a powerful implement. The connection between $\mathcal{L}$ and DDHA gives rise to two  results for the classical representations of $\GL(V)$: (i) A  duality between $\GL(V)\times\Gm$ and DDHA where $\Gm$
is the one-dimensional multiplicative group; (ii) A branching duality formula. With aid of  the above discussion, we further obtain a parabolic Schur-Weyl duality for $\uG\rtimes \Gm$. What is more, the parabolic Schur subalgebra turns out to have only one block. The Cartan invariants for this algebra are precisely determined.
\end{abstract}
\maketitle

\setcounter{tocdepth}{1}\tableofcontents
\begin{center}
\end{center}

\section*{Introduction}
\subsection{} An algebraic group $G$  is called a semi-reductive group if $G$ is a semi-direct product of a reductive closed subgroup $G_0$ and the unipotent radical $U$.
 The study of semi-reductive algebraic groups and their Lie algebras becomes very important to lots of cases when the underground filed is of characteristic $p>0$ (see \cite{OSY}). The examples what we are concerned arise from  restricted simple Lie algebras of Cartan type (see \cite{PS}, \cite{SF} and \cite{Wil}).  In the present paper, we have another natural example.  Let $G=\GL(V)$ and $\nu$ be the natural representation of $G$ on $V$. Then  we have a typical enhanced reductive algebraic group $\uG=G\times_\nu V$, which is a closed subgroup of $\GL(\uV)$ with $\uV$ being a one-dimensional extension of $V$. The enhanced reductive group $\uG$ is naturally a semi-reductive group. Some study of semi-reductive algebraic group will be done in some other place (to see \cite{OSY}).

\subsection{} By the classical Schur-Weyl duality, the study of polynomial representations of general linear groups produces Schur algebras.  Precisely, for a given infinite filed $\bbf$, and $G=\GL(n,\bbf)$ the general linear group over $\bbf$, the Schur algebra $S_\bbf(n,r)$ is exactly $\End_\bbf(E^{\otimes r})^{\fsr}$ for $E=\bbf^n$.
The clear structure makes Schur algebras become powerful and tractable in the study of polynomial representations of $\GL(n,\bbf)$ (see \cite{Gr}, \cite{Ma}, {\sl{etc}}.). 

By analogy of this, the tensor representations of an enhanced group $\uG$ naturally produce the so-called enhanced Schur algebra $\cale_\bbc(n,r)$  which are the algebras generated by the image of $\uG$ in the $r$th tensor representation on $\uvtr$ for $V=\bbc^n$, repsectively (see \S\ref{en sch sec}). Related with $\uG$, there is canonically a short exact sequence of closed subgroups of $\GL_{n+1}$: $\uG\hookrightarrow \uG\rtimes\Gm\twoheadrightarrow \Gm$. And $\uG\rtimes\Gm$ is actually a parabolic subgroup of $\GL_{n+1}$ associated with the Levi subgroup $\GL_n\times \Gm$.
Correspondingly, the so-called Levi and parabolic Schur algebras $\mathscr{L}(n, r)$ and $\mathscr{P}(n, r)$ naturally arise, which are by definition the image of $\GL_n\times \Gm$ and $\uG\rtimes\Gm$ in $\End_{\bbf}(\uvtr)$, respectively.

For simplicity, we work with $\bbf=\bbc$ in the main body of the text. Then we can realize $\cale(n,r):=\cale_\bbc(n,r)$ and $\calp(n,r)$ as subalgebras of $S_\bbc(n+1,r)$. In the case $\bbf=\bbc$, $\GL(n,\bbc)$ will be simply denoted by $\GL_n$.

\subsection{} The main purposes are double. One is to develop the Levi, and parabolic  Schur algebras and the related representations. The another one is to investigate
dualities of variant groups and algebras in the enhanced tensor representations. In this course, an important algebraic model of so-called degenerate double Hecke algebras\footnote{In the literature, double affine Hecke algebras (DAHA) arising from the study of KZ equations are familiar to authors ({\sl{ref}}. \cite{Ch}). Our DDHA  have no direct relation with the former.} (DDHAs for short)  is introduced, in the same spirit of  degenerate affine Hecke algebras and the related (see \cite{Dr} and \cite{AS}).
 Roughly speaking, the DDHA $\calh_r$ of type $A_{r-1}$ is a combination of varieties of subalgebras generated compatibly  by the group algebras $\bbc\fsr$ and $\bbc\fsl$ for all positive integers $l$ not bigger than $r$. The $l$th DDHA $\calh^l_r$ is an associative algebra defined by generators $\bs_i$, and $\bx_\sigma$, $i=1,\ldots,r-1$, and $\sigma\in\fsl$, and by relations:
\begin{align*}
\bs_i\bx_{\sigma}=\bx_{s_i\circ\sigma},\;\; \bx_\sigma \bs_i=\bx_{\sigma\circ s_i} \mbox{ for }\sigma\in\fsl, i<l;
\end{align*}
\begin{align*}
\bs_i\bx_{\sigma}=\bx_\sigma=\bx_\sigma \bs_i \mbox{ for }\sigma\in\fsl, i>l.
\end{align*}
together with the defining relations of  $\fsr$ and of $\fsl$ (see (\ref{ddha 1})-(\ref{ddha 6}) for the complete defining relations of a DDHA). This is an infinite-dimensional algebra. Nevertheless, the $l$th DDHA $\calh^l_r$ naturally arises from the tensor representation $(\bbc^{n+1})^{\otimes r}$ over $\GL_n$. The core operators in $\End_\bbc((\bbc^{n+1})^{\otimes r})$ for $l< r$ and $\sigma\in\fsl$, are presented as below
$$x_\sigma=\Psi_l(\sigma)\otimes \id^{\otimes r-l} $$
where $\Psi_l(\sigma)$ is a position permutation by $\sigma$ in the first $l$th tensor factors (see the paragraph around (\ref{sigma operator}) for the precise meaning). From those $x_\sigma$ and all usual operators $s_i$ interchanging the $i$th and $(i+1)$th factors in $(\bbc^{n+1})^{\otimes r}$, $\calh^l_r$ naturally arises. Correspondingly, $(\bbc^{n+1})^{\otimes r}$ becomes a natural module over $\calh_r$. This representation is denoted by $\Xi$. Then $\Xi(\calh_r)$ becomes finite-dimensional, denoted by $D(n,r)$.
On the enhanced-tensor representation space $\uvtr$,  the role of pair $(\calh_r, D(n,r))$  is somewhat a counterpart of the one of pair $(\GL(V), S_\bbc(n,r))$ when we consider the usual tensor representation space $V^{\otimes r}$.

The above DDHAs  are powerful ingredients   in the course of establishing dualities for the enhanced tensor representations.

\subsection{} Below are the main results.

\begin{theorem}
{\rm(1)} (Theorem \ref{enh sch stru}) The parabolic Schur algebra $\calp(n,r)$ has a basis  $\{\xi_{\tilde\pi_\bs,\bj}\mid (\tilde\pi_\bs,\bj)\in E\}$, and
$$\dim\calp(n,r)=\sum\limits_{k=0}^r{{n+k-1}\choose k}{{n+k}\choose k}.$$
There is a set of canonical generators $\{\theta_{\bs,\bt},\xi_{\bi,\bj}\mid (\bs,\bt)\in\Lambda, (\bi,\bj)\in D\slash \fsr\}$ for $\calp(n,r)$.

{\rm(2)} (Theorems \ref{thm:irr IPM csf} and \ref{thm: Cartan inv}) The parabolic Schur algebra $\calp(n,r)$ has one block. The isomorphism classes of irreducible modules and the indecomposable projective modules (PIM for short) are parameterized by the set of dominant weights $\Lambda^+:=\bigcup_{l=1}^r\Lambda^+(n,l)$, respectively.
For an irreducible module $D_{\gamma'}$ and a PIM $P_\gamma$ with $\gamma,\gamma^{\prime}\in \Lambda^+$, set $a_{\gamma,\, \gamma^{\prime}}:=(P_\gamma: D_{\gamma^{\prime}})$, and $\ell(\gamma)=\#\mathfrak{S}_n.\gamma$.  Then the Cartan invariants are presented as
\begin{equation*}\label{comp. number}
a_{\gamma,\ \gamma^{\prime}}=\begin{cases}
1, &\text{if}\,\, \gamma^{\prime}=\gamma;\cr
\ell(\gamma^{\prime}), &\text{if}\,\,|\gamma^{\prime}|<|\gamma|;\cr
0, &\text{otherwise}.
\end{cases}
\end{equation*}
\end{theorem}

With aid of the structural description of the Levi subalgbra $\scrl(n,r)$ of the parabolic Schur algebra $\calp(n,r)$, we establish a duality between DDHA and $\GL_n\times\Gm$
in $\End_\bbc((\bbc^{n+1})^{\otimes r})$ and its consequence a ``duality branching formula" in dimensions. As its application, we finally obtain the  Levi/Parabolic Schur-Weyl duality stated below.

\begin{theorem} For $\uV=\bbc^{n+1}$,  denote by $(\uvtr, \Phi)$ the tensor representations of $\GL_{n+1}=\GL(\uV,\bbc)$,  and by $(\uvtr,\Psi)$ the permutation representations of the symmetric group $\fsr$.

{\rm(1)}  (Theorem \ref{res swd}) Let $\calh_r$ be the degenerate double Hecke algebra associated with $\fsr$. Then the following Levi Schur-Weyl duality holds.
\begin{align*}
 \End_{\bbc\Phi(\GL_n\times \Gm)}(\uvtr)&=\Xi(\calh_r);\cr
 \End_{\Xi(\calh_r)}(\uvtr)&=\bbc\Phi(\GL_n\times \Gm).
  \end{align*}

{\rm(2)} (Corollary \ref{coro bran d f})  For $\lambda\in \Lambda^+(n,l)$ (a partition of $l$ with $n$ parts),  denote by $S^\lambda_l$ the irreducible module of $\fsl$ corresponding to $\lambda$.
The following dimension formula can be regarded as  a duality to the classical branching rule. For   $l\in \{1,2,\ldots,r\}$, and $\lambda\in \Lambda^+(n,l)$,

\begin{equation*}
\sum\limits_{\overset{\mu\in \Lambda^+(n+1,r)}{\lambda \text{ interlaces }\mu}}\dim S^{\mu}_r={r\choose l}\dim S^{\lambda}_l
\end{equation*}

{\rm(3)}  (Theorem \ref{enh sch thm}) $\End_{\bbc\Phi(\uG\rtimes\Gm)}(\uvtr)=\Xi(\calh_r)^V$ for $G=\GL(V)$.
\end{theorem}

It is worth mentioning that the study of invariants beyond reductive groups is a challenge (see \cite{Gr1}-\cite{Gr3}).  Consequently the invariant property of $\uG\rtimes\Gm$ as above has its own interest.

In the same spirit, one can consider more general case for Levi- and parabolic subgroups of $\GL_{n+1}$.  

\vskip5pt
\subsection{} Our paper is organized as follows. In the first section, we introduce some basic notions and notations for semi-reductive groups. Then we give some fundamental properties of semi-reductive groups. In the second section, we introduce the Levi/Parabolic Schur algebras, and investigate their structure. In the third section, we introduce the degenerate double Hecke algebras and demonstrate their representation meaning in the enhanced tensor products.
The fourth section is devoted to the proof of Theorem \ref{res swd} and Corollary \ref{coro bran d f}. With the above, in the fifth section  we  first prove the Levi and parabolic Schur Weyl dualities and then give some tensor invariants. In the concluding section, we investigate the representations of the parabolic Schur algebras, obtaining the results on the blocks and Cartan invariants.

\vskip5pt
\subsection{} As to other aspects, there will be some investigations somewhere else (to see \cite{LW}, \cite{SXY}, \cite{XZ}, {\sl{etc}}.). The parabolic Schur algebras can be defined in prime characteristic. Their modular representations will be an interesting topic in the future. It is worth mentioning that the nilpotent cone of $\underline{\gl(V)}:=\Lie(\uG)$ for $G=\GL(V)$ is the same as the enhanced nilpotent cone studied by Achar-Hendersen in \cite{AH}. Adjoint nilpotent orbits of $\uG$ in $\underline{\gl(V)}$ are compatible with enhanced nilpotent orbits studied in \cite{AH}. With respect to those, there are some interesting phenomenon parallel to the $\GL(V)$-theory in prime characteristic, including Jantzen's realization of support varieties of Weyl modules by the closures of some nilpotent orbits (see \cite{Jan2}).

\section{Semi-redutive groups and semi-reductive Lie algebeas}
In this section, all vector spaces and varieties are  over a field $\bbf$ which stands for either the complex number field $\bbc$, or an algebraically closed field $\bk$ of characteristic $p>0$. 
\subsection{Notions and notations}
\begin{definition}
An algebraic group $G$ over $\bbf$ is called semi-reductive if $G=G_0\ltimes U$ with $G_0$ being a reductive subgroup, and $U$ the unipotent radical. Let $\ggg=\Lie(G)$, and $\ggg_0=\Lie(G_0)$ and $\uuu=\Lie(U)$, then $\ggg=\ggg_0\oplus\uuu$.
\end{definition}

\begin{example}\label{ex3} (Enhanced reductive algebraic groups) Let $G_0$ be a connected reductive algebraic group over $\bbf$, and $(M,\rho)$ be a finite-dimensional rational representation of $G_0$ with representation space $M$ over $\bbf$.    Consider the product variety $G_0\times M$. Regard $M$ as an additive algebraic group. The variety $G_0\times M$ is endowed with an enhanced cross product structure denoted by $G_0\times_\rho M$,  by defining for any $(g_1,v_1), (g_2,v_2)\in  G_0\times M$
\begin{align}\label{enhanced product}
(g_1,v_1)\cdot (g_2,v_2):=(g_1g_2, \rho(g_1)v_2+v_1).
\end{align}
Then by a straightforward computation it is easily known that  $\underline{G_0}:=G_0\times_\rho M$ becomes a group with identity $(e,0)$ for the identity $e\in G_0$, and $(g,v)^{-1}=(g^{-1}, -\rho(g)^{-1}v)$.
 And $G_0\times_\rho M$ has a subgroup $G_0$ identified with $(G_0, 0)$ and a subgroup $M$ identified with  $(e, M)$.
Furthermore,  $\underline{G_0}$ is connected since $G_0$ and $M$ are irreducible varieties. We call $\underline{G_0}$ \textsl{ an enhanced reductive algebraic group associated with the representation space $M$}. What is more, $G_0$ and $M$ are closed subgroups of $\underline{G_0}$, and $M$ is a normal closed subgroup. Actually, we have $(g,w)(e,v)(g,w)^{-1}=(e,\rho(g)v)$ for any $(g,w)\in G_0$. From now on,   we will write down $\dot g$ for $(g,0)$ and $e^v$ for $(e,v)$ unless other appointments. It is clear that $e^{v}\cdot e^{w}=e^{v+w}$ for $v,w\in V$.

Suppose $\underline{\ggg_0}=\Lie(\uGz)$. Then  $(M, \mathsf{d}(\rho))$ becomes a representation of $\ggg_0$.  Naturally, $\Lie(\underline{G_0})=\ggg_0\oplus M$, with Lie bracket
 $$[(X_1,v_1),(X_2,v_2)]:=([X_1,X_2], \mathsf{d}(\rho)(X_1)v_2-\mathsf{d}(\rho)(X_2)v_1),$$
 which is called an enhanced reductive Lie algebra.

Clearly, $\uGz$ is a semi-reductive group with $M$ being the unipotent radical.
\end{example}

\section{Enhanced tensor representations and parabolic Schur algebras}

 From now on,  the ground field $\bbf$ will be $\bbc$. We suppose $G=\GL(V)$ over $\bbc$, and suppose that $\underline G:=G\times_\nu V$ is an enhanced group of $G$ associated with the natural representation $\nu$ on $V$. All representations for algebraic groups are always assumed to be rational ones.

\subsection{Enhanced natural modules}   An irreducible $G$-module becomes naturally an irreducible module of $\underline{G}$ with trivial $V$-action. The isomorphism classes of irreducible rational representations of $\uG$ coincide with the ones of $G$.

Denote by $\underline{V}$ the one-dimensional extension of $V\cong \bbc^n$ via the vector $\eta$, i.e. $\underline{V}=V\oplus \bbc \eta\cong \bbc^{n+1}$. We will call $\uV$ the enhanced space (of $V$).
Naturally  $\uV$  becomes a $\uG$-module which is defined for any $\ug:=(g,v)\in \uG$ and $\uu=u+a\eta\in \uV$ with $g\in G$, $v,u\in V$ and $a\in \bbc$, via
\begin{align}\label{enhanced action}
\ug.\uu=\nu(g)u+av+a\eta.
\end{align}
It is not hard to see that this module is a rational module of  $\uG$. The corresponding rational representation of $\uG$ is denoted by $\underline{\nu}$, which gives rise to a short exact sequence of $\uG$-modules
\begin{align}\label{short exact}
V\hookrightarrow \uV \twoheadrightarrow \bbc
\end{align}
where $\bbc$ means the one-dimensional trivial $\uG$-module. The following fact is clear.

\begin{lemma}\label{sub} The enhanced reductive algebraic group $\underline{\GL(V)}$ is a closed subgroup of $\GL(\uV)$.
\end{lemma}

\subsection{Classical Schur-Weyl duality}
Recall the classical Schur-Weyl duality. The natural representation of $\GL(V)$ on $V$ gives rise to a $G=\GL(V)$-module on the tensor product $\vtr$ for any given positive integer $r$,  with diagonal $g$-action for any $g\in G$. The corresponding representation is denoted by $\phi$. This means
$$\phi(g)(v_1\otimes v_2\otimes\cdots\otimes v_r)=\nu(g)v_1\otimes\nu(g)v_2\otimes\cdots\otimes \nu(g)v_r$$
for any $g\in \GL(V)$ and any monomial tensor product $v_1\otimes\cdots\otimes v_r\in \vtr$.

In the meanwhile, $\vtr$ naturally becomes  a $\fsr$-module with permutation action. The corresponding $\fsr$-representation on $\vtr$ is denoted by $\psi$.  This means that for any $\sigma\in \fsr$,
  \begin{align}\label{sym action on tensor}
  \psi(\sigma)(v_1\otimes v_2\otimes\cdots\otimes v_r)=v_{\sigma^{-1}(1)}\otimes v_{\sigma^{-1}(2)}\otimes\cdots \otimes v_{\sigma^{-1}(r)}.
  \end{align}

 The classical Schur-Weyl duality shows that the images of $\phi$ and $\psi$ are double centralizers in $\End_\bbc(\vtr)$, i.e. for $G=\GL(V)$
\begin{align}\label{sw duality}
&\End_{G}(\vtr) =\bbc\psi(\fsr)\cr
 & \End_{\fsr}(\vtr)             =\bbc\phi(G)
\end{align}
Here $\bbc\psi(\fsr)$ and $\bbc\phi(G)$ stand for the subalgebras of $\End_\bbc(\vtr)$ generated by $\psi(\fsr)$ and $\phi(G)$ respectively.

\subsection{Enhanced tensor representations} Now consider the $r$th tensor product $\uvtr$ for a fixed positive integer $r$, which becomes a $\uG$-modules by diagonal action.

From the classical Schur-Weyl duality (see (\ref{sw duality})), $\vtr$ has the following decomposition as $\GL(V)\times \fsr$-modules
\begin{align}\label{tensor dcp}
 \vtr= \bigoplus_{\lambda}L^\lambda\otimes D^\lambda
\end{align}
where $\lambda$ in the sum  ranges over the set of partitions of $r$ with $n$ parts (zero parts are allowed), both $L^\lambda$ and $D^\lambda$ are the irreducible highest weight module of $\GL(V)$ and the irreducible module of $\fsr$ respectively, associated with the partition  $\lambda$ (see \cite{GW}).

We now consider the composition factors of $\uvtr$ as a $\uG$-module.
Keeping  (\ref{short exact}) and (\ref{tensor dcp}) in mind, one easily has the following.

\begin{proposition} As a $\uG$-module, the following formula holds in the Grothendieck group of the rational $\uG$-module category
$$[\uvtr]=\sum_{\mu} c_\mu [L^\mu]$$
where the sum is over partitions $\mu$ of $l\in \{0,1,\ldots,r\}$ of length $\leq n$, and $$c_\mu={r\choose l}\dim S^\mu_l.$$

\end{proposition}

\subsection{A preliminary to enhanced Schur-Weyl dualities}
Denote by $\Phi$ the representation of $\uG$ on the $r$-tensor product module $\uvtr$. Then the following question is naturally raised.

\begin{question}\label{ques} What is the centralizer of $\bbc\Phi(\uG)$ in $\End_\bbc(
\uvtr)$? What role does $\fsr$ play  in the enhanced case as in the classical Schur-Weyl duality?
\end{question}

It is not trivial to give a complete answer which is left till \S\ref{sec 5} (see Theorem \ref{enh sch thm}). Here we give some preliminary investigation.
Note that $\GL(\uV)$  contains a closed subgroup $\uG$ (Lemma \ref{sub}).  According to the classical Schur-Weyl duality, $\GL(\uV)$ and $\fsr$ are dual in the sense of the double centralizers in $\End_\bbc(\uvtr)$ (we will still denote by $\Phi$ the natural representation of $\GL(\uV)$ on $\uvtr$ if the context is clear). So the significant investigation of Question \ref{ques}  is to decide the centralizer of $\bbc\Phi(\uG)$ in $\End_\bbc(\uvtr)$.
 For this, we first have the following observation

\begin{lemma} \label{restriction}
 For any $\xi\in \End_{\uG}(\uvtr)$, $\xi(\vtr)\subset \vtr$.
\end{lemma}

\begin{proof} For any nonzero $w\in \vtr$, we want to show that $\xi(w)\in \vtr$. We might as well suppose $\xi(w)\ne 0$. Then we write
\begin{equation}\label{expression of xi(w)}
\xi(w)=\sum_{i=1}^t \uu_{i1}\otimes\uu_{i2}\otimes \cdots \otimes \uu_{ir}\in \uvtr
\end{equation}
with $\uu_{ij}=u_{ij}+a_{ij}\eta$, where $u_{ij}\in V$, and $a_{ij}\in \bbc$, and
the number $t$ of summands in (\ref{expression of xi(w)}) is minimal among all possible expression of $\xi(w)$.

By the assumption, $\xi(\Phi(\ug) w)=\Phi(\ug)\xi(w)$ for any $\ug:=(g,v)\in\uG$ with $g\in G$ and $v\in V$. In particular, we take some special element $(e,v)\in \uG$, denoted by  $e^v$. Then we have an equation
 \begin{align}\label{ev equ}
 \xi(\Phi(e^v)w)=\Phi(e^v)\xi(w).
 \end{align}
  By (\ref{enhanced action}), $\underline{\nu}(e^v)(\uu_{ij})=u_{ij}+a_{ij}v+a_{ij}\eta$. So
 we have
 \begin{align*}
\Phi(e^v)\xi(w)=\sum_{i} (u_{i1}+a_{i1}v+a_{i1}\eta)\otimes\cdots\otimes (u_{ir}+a_{ir}v+a_{ir}\eta).
 \end{align*}
On the other hand, by (\ref{enhanced action}) again we have $\underline{\nu}(e^v)u=u$ for any $u\in V$. Hence, $\Phi(e^v)w=w$.
We have
\begin{align*}
\xi(\Phi(e^v)w)&=\sum_{i} \uu_{i1}\otimes \cdots \otimes \uu_{ir}.
\end{align*}
By comparing both sides of (\ref{ev equ}), the arbitrariness of $v$ leads to all $a_{ij}$ being equal to zero.   The proof is completed.
\end{proof}

\begin{proposition}\label{enhanced swty}
 Denote by $\Psi$ the permutation representation of $\fsr$ on $\uvtr$ defined as in (\ref{sym action on tensor}). The following statements.
  \begin{itemize}
 \item[(1)] $\bbc\Psi(\fsr)$ is a subalgebra of  $\End_{\uG}(\uvtr)$.
\item[(2)] There is a surjective homomorphism of algebras
$$\End_{\uG}(\uvtr)\twoheadrightarrow\bbc\Psi(\fsr).$$
\end{itemize}
\end{proposition}

\begin{proof} (1) Note that $\uG$ is a subgroup of $\GL(\uV)$ and then $\End_{\GL(\uV)}(\uvtr)\subset\End_{\uG}(\uvtr)$.  This statement follows from the classical Schur-Weyl duality with respect to $\GL(\uV)$ and $\fsr$.

 (2) By Lemma \ref{restriction}, we can define a map $\res$ from $\End_{\uG}(\uvtr)$ to $\End_{G}(\vtr)$ by sending $\xi$ to $\xi|_{\vtr}$.
Then it is easily seen that $\res$ is an algebra homomorphism. Furthermore, we assert that $\res$ is surjective. Actually, by the classical Schur-Weyl duality we have $\End_{G}(\vtr)=\bbc\Psi(\fsr)$. So any element of $\End_{G}(\vtr)$ can expressed as $\sum_{i}a_i\sigma_i$ which is a finite $\bbc$-linear combination of some $\sigma\in \fsr$. Take $\xi$ to be a morphism in $\End_\bbc(\uvtr)$ sending  $\uu_1\otimes\cdots \otimes\uu_r\in \uvtr$ to
$$ \sum_i a_i \uu_{\sigma_i^{-1}(1)}\otimes\cdots\otimes \uu_{\sigma_i^{-1}(r)}.$$
Then $\xi$ is $\uG$-equivariant. And $\res(\xi)$ is exactly $\sum_{i}a_i\sigma_i\in \End_{G}(\vtr)$. So the assertion is proved.
\end{proof}

\subsection{Levi and parabolic subalgebras of $S(n+1,r)$} \label{en sch sec}
 From now on, we always assume  $\dim V=n$.
 Set $\cala:= \End_\bbc(\uvtr)$, and $\cals=\bbc\Phi(\GL(\uV))$, $\cale:=\bbc\Phi(\uG)$ and $\calc:=\bbc\Psi(\fsr)$ for $\uG=\GL(V)\times_\nu V$. Call $\cale$ an enhanced Schur algebra which reflects the polynomial representations of degree $r$ for the enhanced reductive algebraic group $\uG$. Obviously, $\cale$ is a subalgebra of the semi-simple algebra $\cals$. 
Note that $\bbc\Phi(\GL(\uV))=\End_\bbc(\uvtr)^{\fsr}$ is actually the classical Schur algebra $S(n+1,r)$ (see \cite{Gr}). Lemma \ref{sub} still holds in any case. Correspondingly,   we denote $\cale$ by $\cale(n,r)$ more precisely.

By the above arguments, we have a sequence of the following subalgebras in $\cala$:
$$S(n,r)\subset \cale(n,r)\subset S(n+1,r). $$ Furthermore, we will have more interesting and related subalgebras: Levi Schur algebras and parabolic Schur algebras.

\subsubsection{} \label{classical sch} Let us first recall some facts on the classical Schur algebras. The readers refer to \cite{Gr} or \cite{DDPW} for the details.

For a given positive integer $m$, set $\um=\{1,2,\ldots,m\}$. Denote by $A(m,r)$ the space consisting of the elements expressible as polynomials which are homogeneous of degree $r$ in the polynomial function $c_{i,j}$ $(i,j\in\um)$ on $\GL(m,\bbc)$. Then  $A(m,r)$ has a basis (modulo the order of factors in the monomials)
$$c_{\bi,\bj}=c_{i_1j_1}c_{i_2j_2}\cdots c_{i_rj_r}$$
for $\bi=(i_1,\ldots,i_r), \bj=(j_1,\ldots,j_r)\in \um^r$. The symmetric group $\fsr$ acts on the left on $\um^r$ by $\sigma.\bi=(i_{\sigma^{-1}(1)},\ldots,i_{\sigma^{-1}(r)})$. Furthermore, $\fsr$ act also on the set $\um^r\times \um^r$ by $\sigma.(\bi,\bj)=(\sigma.\bi,\sigma.\bj)$. So we can define an equivalence relation $\sim$ on $\um^r\times \um^r$, $(\bi,\bj)\sim (\bk,\bl)$ if and only if $(\bk,\bl)=\sigma.(\bi,\bj)$ for some $\sigma\in \fsr$. The number of the equivalence classes (or the orbits) in $\um^r$ under such a $\fsr$-action is just $\sf{d}_{m,r}:={{m^2+r-1}\choose {r}}$. Then $A(m,r)$ becomes a coalgebra of dimension $\sf{d}_{m,r}$ with the coproduct $\delta$ and the counit $\varepsilon$ as below
\begin{align}
\Delta: A(m,r)&\rightarrow A(m,r)\otimes A(m,r), \cr
 c_{\bi,\bj}&\mapsto \sum_{\bk}c_{\bi,\bk}\otimes c_{\bk,\bj};
\end{align}
and $\varepsilon(c_{\bi,\bj})=\delta_{\bi,\bj}$.

Alternatively, the classical Schur algebra $S(m,r)$ can be defined as the dual of the $A(m,r)$.  So $S(m,r)$ has basis $\{\xi_{\bi,\bj}\mid (\bi,\bj)\in \um^r\times\um^r\slash \sim\}$ dual to the basis $\{c_{\bi,\bj}\mid (\bi,\bj)\in \um^r\times\um^r\slash \sim\}$ of $A(m,r)$. This means
\begin{align*}
\xi_{\bi,\bj}(c_{\bk,\bl})=\begin{cases} 1, &\mbox{ if } (\bi,\bj)\sim (\bk,\bl);\cr
0,  &\mbox{ otherwise}.
\end{cases}
\end{align*}
 Furthermore, $S(m,r)$ is an associative algebra of dimension $\sf{d}_{m,r}$  with the following multiplication rule
\begin{itemize}
\item[(S1)] $\xi_{\bi,\bj}\xi_{\bk,\bl}=\sum_{\bp,\bq} a_{\bi,\bj,\bk,\bl,\bp,\bq}\xi_{\bp,\bq}$, where  $a_{\bi,\bj,\bk,\bl,\bp,\bq}$ equals the number of the elements $\bs\in \um^r$ satisfying $(\bi,\bj)\sim (\bp,\bs)$ and $(\bk,\bl)\sim (\bs,\bq)$.
\item[(S2)] $\xi_{\bi,\bj}\xi_{\bk,\bl}=0$ unless $\bj\sim \bk$.
\item[(S3)] $\xi_{\bi,\bi}\xi_{\bi,\bj}=\xi_{\bi,\bj}=\xi_{\bi,\bj}\xi_{\bj,\bj}$.
\end{itemize}

\subsubsection{A key lemma for Schur algebras} Let $\Phi_m:\GL(m,\bbc)\rightarrow \End_\bbc(\vtr)$ for $V=\bbc^m$. By the classical Schur algebra theory, we can identify the image of $\Phi_m$  with $S(m,r)$. So we can write the image precisely.

In general, take $g$ from $\GL(m,\bbc)$ with $g=(g_{pq})_{m\times m}$. Then
\begin{align}\label{apq}
\Phi_m(g)=\sum_{(\bp,\bq)\in \um^r\times\um^r\slash\fsr}a_{\bp\bq}\xi_{\bp,\bq} \mbox{ with } a_{\bp\bq}=\prod_{i=1}^r g_{p_iq_i} \mbox{ for } \bp=(p_1,\ldots,p_r), \bq=(q_1,\ldots,q_r).
\end{align}

Set $I_{g}:=\{(\bp,\bq)\in \um^r\times\um^r\slash\fsr\mid a_{\bp \bq}\ne 0\}$. Then  $\Phi_m(g)$ can be expressed as
\begin{align}
 \Phi_m(g)=\sum_{(\bp,\bq)\in I_{g}} a_{\bp\bq}\xi_{\bp,\bq}.
\end{align}

Set $I^0(g)=\{(p,q)\in \um\times\um\mid g_{pq}\ne0\}$, and set $\bbc^\times=\bbc\backslash\{0\}$.
Denote $B(g):=\{h\in\GL(m,\bbc)\mid I^0(h)=I^0(g)\}$. Then $B(g)$  can be regarded as an intersection of a non-empty open subset in $\bbc^{\#I^0(g)}$ with $(\bbc^\times)^{\#I^0(g)}$. So it is still a nonzero open subset of $\bbc^{\#I^0(g)}$.

Fix an order for all elements of the set $I_g$. Set $\lambda(g)=\#I_g$.
Then we can talk about the matrix for $\Phi_m(g)$ for $g\in \GL(m,\bbc)$.
 The following observation is fundamental,  which will be important to the sequel arguments.

\begin{lemma} \label{general GL lem} Suppose $S(m,r)$ is the classical Schur algebra associated with $\GL(m,\bbc)$ and  degree $r$.  Keep the notations as above. The following statements hold.
 \begin{itemize}
 \item[(1)] For any given $g\in \GL(m,\bbc)$, there exist  $\lambda(g)$ elements $h^{(i)}\in B(g)$, $i=1,\dots,\lambda(g)$ such that the matrix $(a^{(i)}_{\bp\bq})_{\lambda(g)\times \lambda(g)}$
is invertible, where $a^{(i)}_{\bp\bq}$ is defined in the same sense as in
(\ref{apq}) with respect to $h^{(i)}$.

\item[(2)] Consequently, for any $(\bp,\bq)\in I_\ggg$, $\xi_{\bp,\bq}=\sum_{i=1}^{\lambda(g)} c_i\Phi_m(h^{(i)})$ for some $c_i\in \bbc$.

\item[(3)] Furthermore, for any $(\bp,\bq)\in \um^r\times\um^r$, there exists $g\in\GL(m,\bbc)$ such that $(\bp,\bq)\in I_\ggg$. Therefore, (2) is valid for any basis element $\xi_{\bp,\bq}$ of $S(m,r)$.
\end{itemize}
\end{lemma}

\begin{proof}
(1) For  $g=(g_{ij})_{m\times m}\in \GL(m,\bbc)$, we set $l=\lambda(g)$ and $I_g=\{\tau_1,\cdots, \tau_l\}$. Then $$\Phi_m(g)=\sum_{k=1}^{l}a_{\tau_k}\xi_{\tau_k}\,\text{ with}\, a_{\tau_k}=\prod_{i=1}^r g_{p_iq_i}\,\text{if}\, \tau_k=(\bp,\bq)\in \um^r\times\um^r\slash\fsr.$$ For $s\in\mathbb{N}$, set
$$\mathfrak{P}_s=\{(c_{ij})_{m\times m}\in(\mathbb{C}^\times)^{m^2}\mid (c_{ij}^sg_{ij})_{m\times m}\in\GL(m,\bbc)\}.$$
Each $\mathfrak{P}_s$ is a nonempty open subset of $\mathbb{A}^{m^2}$. In particular, $\cap_{s=1}^{l-1}\mathfrak{P}_s\neq \emptyset$. For $\bc=(c_{ij})_{m\times m}\in\cap_{s=1}^{l-1}\mathfrak{P}_s$, let $g_{\bc,s}=(c_{ij}^sg_{ij})_{m\times m}\in\GL(m,\bbc)$ for $s=1,\cdots, l-1$. Then
$$\Phi_m(g_{\bc,s})=\sum_{k=1}^{l}f_k(c_{ij})^sa_{\tau_k}\xi_{\tau_k}$$ where $f_k(x_{ij})$ is a monomial of degree $r$ over the $n^2$ variables $x_{ij}\,(1\leq i, j\leq n)$ and $f_s(x_{ij})\neq f_t(x_{ij})$ for any $s\neq t$. Therefore,
$$\big(\bigcap\limits_{j=1}^{l-1}\mathfrak{P}_j\big)\cap\big(\bigcap\limits_{1\leq s\neq t<l}\mathfrak{X}_{st}\big)\neq \emptyset, $$
where
$$\mathfrak{X}_{st}=\{(c_{ij})_{m\times m}\in(\mathbb{C}^\times)^{m^2}\mid f_s(c_{ij})\neq f_t(c_{ij})\}\,\,\text{for}\,\,1\leq s\neq t\leq l.$$
Now take
$$\bc=(c_{ij})_{m\times m}\in\big(\bigcap\limits_{j=1}^{l-1}\mathfrak{P}_j\big)\cap\big(\bigcap\limits_{1\leq s\neq t<l}\mathfrak{X}_{st}\big).$$
Set $h^{(1)}=g$ and $h^{(i)}=g_{\bc, i-1}$ ($2\leq i\leq l$). Then the corresponding matrix $(a^{(i)}_{\bp\bq})_{\lambda(g)\times \lambda(g)}$ forms a Vandermonde one. It is desired.

(2) follows directly from (1).

(3) Take $c_1,\cdots, c_m\in\mathbb{C}\backslash\{0\}$ such that $c_i\neq c_j$ for any $1\leq i\neq j\leq m$, and $g=(g_{ij})$ with $g_{ij}=c_j^{i-1}$ for $1\leq i, j\leq m$. Then $g\in\GL(m,\bbc)$ and $(\bp,\bq)\in I_\ggg$ for any $(\bp,\bq)\in \um^r\times\um^r$.

\end{proof}

\subsubsection{}\label{enhanced space}  Now we look for a set of generators of $\cale(n,r)$. Generally, we denote by $\calf(\GL_m)$ the set of functions on $\GL_m:=\GL(m,\bbc)$ for a positive integer $m$. Then
we can naturally regard $\calf(\GL_{n})$ as a subset of $\calf(\GL_{n+1})$,  $A(n,r)$ as a sub-coalgebra of $A(n+1,r)$ and $S(n,r)$ as a subalgebra of $S(n+1,r)$.
Consider a map
$$\theta: \cale(n,r)\rightarrow S(n+1,r)$$
sending $\Phi(\ug)$ to $\theta_{\ug}$ for $\ug=(g,v)$. This is an algebra homomorphism.
Note that $(g,v)=(g,0)(e, g^{-1}v)$. So $\theta_{\ug}=\theta_{(g,0)}\theta_{(e,g^{-1}v)}$.
By the definition of enhanced groups (see (\ref{enhanced product})),  under $\theta$  we can identify  $\Phi(\GL(V))$ with $S(n,r)$. Still set $e^v=(e,v)$ for $v\in V$. Denote $\Omega(n,r)$ the subalgebra of $\cale(n,r)$ generated by $\theta_{e^v}$ with $v$ ranging over $V$.
Then the first question is to understand  $\theta_{e^v}$ and the subalgebra $\Omega(n,r)$ generated by them.

In the standard basis elements $c_{\bi,\bj}$  of $A(m,r)$, the factor $c_{i,j}$ stands for the function which associates to each $g\in \GL_m$ its $(i,j)$-coefficient $g_{ij}$.

Let us turn to $\GL(\uV)$. Firstly, from now on we will always fix  a basis
 \begin{align*}
  \{\eta_i\mid i=1,2,\ldots,n+1\} \mbox{ for }\uV \mbox{ with }\eta_{n+1}=\eta; \eta_1,\ldots,\eta_n\in V,
  \end{align*}
  and then  identify $\GL(\uV)$ with $\GL_{n+1}$. In particular, $e^v$ becomes the following $(n+1)\times (n+1)$ matrix
\begin{equation*}
\left( \begin{array}{ccccccc}
 1 & 0 & 0 &\cdots &0&0& a_1 \cr
0 &1 &0 &\cdots &0&0&a_2\cr
0&0&1&\cdots &0&0&a_3 \cr
\vdots&\vdots&\vdots&\vdots&\vdots&\vdots&\vdots\cr
0&0&0&\cdots &0&1&a_n\cr
0&0&0&\cdots &0&0&1
\end{array}\right).
\end{equation*}
where $v=\sum_{i=1}^na_i\eta_i\in V$.

In the sequel, we always set $\caln:=\{1,2,\ldots, n+1\}$ (note that, in general, we set $\ul=\{1,2,\ldots,l\}$ for the positive integer $l$ throughout the paper, only with an exception for $n+1$ because we intend to stress this special situation in our paper).
For $\bs=(s_1,\ldots,s_n)\in \bbn^n$, $\bt=(t_1,\ldots,t_{n+1})\in \bbn^{n+1}$ we set $|\bs|=\sum s_i, |\bt|=\sum_i t_i$. Furthermore,
we denote by  $(...i^d...)$ an $r$-tuple in $\caln^r$, where $i$ appears $d$ times continuously.   By a straightforward  computation we have
\begin{align}\label{cobasis}
\theta_{e^v}=\sum_{(\bs,\bt)} (\prod_{k=1}^n a_k^{s_k})\xi_
{(1^{s_1}2^{s_2}\ldots n^{s_n}1^{t_1}2^{t_2}\ldots n^{t_n}(n+1)^{t_{n+1}}), ((n+1)^{|\bs|} 1^{t_1} 2^{t_2}\ldots n^{t_n}(n+1)^{t_{n+1}})}.
\end{align}
Here $(\bs,\bt)$ in the sum  runs through the range $\Lambda:=\{(\bs,\bt)\in\bbn^n\times \bbn^{n+1}\mid |\bs|+|\bt|=r\}$, and we appoint $0^0$ to be $1$ if it appears in the coefficient $a_\bs:=\prod_{k=1}^n a_k^{s_k}$.  Denote by  $\theta_{\bs,\bt}$ the summand in the expression in (\ref{cobasis}) corresponding to $(\bs,\bt)$.

Set $\bbn^n_r:=\{\bs \in\bbn^n\mid|\bs|\leq r\}$. For $\bs\in\bbn^n_r$, set
$\Lambda_\bs:=\{\bt\in\bbn^{n+1}\mid(\bs,\bt)\in\Lambda\}$, and set
$$\Theta_\bs:=\sum_{\bt\in \Lambda_\bs}\theta_{\bs,\bt}.$$
 Then
$$\theta_{e^v}=\sum_{{\bs}\in \bbn^n_r}a_\bs \Theta_{\bs}.$$

\begin{lemma}\label{gen V} The subalgebra $\Omega(n,r)$ has a basis $\{\Theta_\bs\mid \bs\in\bbn^n_r\}$. Therefore,  $\dim\Omega(n,r)=\sum_{k=0}^{r}{{n+k-1}\choose k}$.
\end{lemma}

\begin{proof}  For $a\in\bbc$, take $v_a=\sum\limits_{i=1}^na^{r^{i-1}}\eta_i\in V$. Then
$$\theta_{e^{v_a}}=\sum\limits_{\bs\in \bbn^n_r}a^{||\bs||}\Theta_{\bs},$$
where $||\bs||=\sum_{i=1}^nr^{i-1}s_i$. Then $||\bs||\neq ||\bs^{\prime}||$ for any distinct $\bs,\bs^{\prime}\in\bbn^n_r$.
Denote $\lambda(n,r)=\#\bbn^n_r$.
Take $c\in \bbc^\times$ and $c$ is not a root of unit, and $a_i=c^i$ for $0\leq i\leq \lambda(n,r)-1$.  This implies that $\Theta_{\bs}$ is in the subspace spanned by $\{\theta_{e^{v_{a_i}}}\mid 0\leq i\leq \lambda(n,r)-1\}$ for any $\bs\in \bbn^n_r$. In particular, $\Theta_{\bs}$ is in the subalgebra generated by $\theta_{e^{v}}$ for $v\in V$. Consequently, $\Omega(n,r)$ coincides with the subalgebra of $\cale(n,r)$ generated by $\theta_{e^{v}}$ for $v\in V$. And  we have $\Omega(n,r)=\sum_{\bs\in \bbn^n_r} \bbc \Theta_{\bs}$.

Next we need to show that all $\Theta_{\bs}$ are linearly independent. For any $\bs\in \bbn^n_r$, take $\bt\in\bbn^{n+1}$ such that $(\bs, \bt)\in\Lambda$. Denote by $c_{\bs,\bt}$ the basis element in $A(n+1,r)$ corresponding to $\xi_{\bs,\bt}$. Then $\Theta_{\bs}(c_{\bs,\bt})=1$, and  $\Theta_{\bs}(c_{\bs',\bt'})=0$ whenever $\bs\ne\bs'$.
For any given combinator equation  $\sum_{\bs}k_{\bs}\Theta_{\bs}=0$, taking the value at $c_{\bs,\bt}$ we have $k_{\bs}=0$. Hence,  those $\Theta_{\bs}$ are proved to be linearly independent, so  that the dimension formula follows.
\end{proof}

\subsubsection{Levi and Parabolic subalgebras of $S(n+1,r)$} \label{para sub} There are canonical imbeddings of algebraic groups  $\GL_n\hookrightarrow\GL_{n+1}$  given by
\begin{equation*}
g\mapsto \diag(g,1):= \left( \begin{array}{cc}
 g & 0\cr
 0&1
\end{array}\right) \;\; \forall g\in\GL(n),
\end{equation*}
 and $\Gm\hookrightarrow\GL_{n+1}$ given by $c\mapsto \diag(1,\ldots,1,c)$ $\forall c\in \Gm=\bbc^\times$, respectively. In this sense, both $\GL_n$ and $\Gm$ can be regarded as closed subgroups of $\GL_{n+1}$.  Clearly, the closed subgroup in $\GL_{n+1}$ generated by $\GL_n$ and $\Gm$  is isomorphic to $\GL_n\times \Gm$. We identify both throughout the paper. In the same sense, we identify the closed subgroup generated by $\uG$ and $\Gm$ with  $\Gm\ltimes \uG$ (note that $\Gm$ normalizes $\uG$).  Denote by $\scrl(n,r)$ the image of $\GL_n\times \Gm$ under $\Phi$ which we call the Levi subalgebra of $S(n+1,r)$. And denote by $\calp(n,r)$ the image of  $\Gm\ltimes \uG$  under $\Phi$, which is called the parabolic subalgebra of $S(n+1,r)$. Clearly, when $r\geq 2$, $\calp(n,r)$ is an extension of $\cale(n,r)$ with the following exact sequence
 \begin{align}\label{ext exact}
   \cale(n,r)\hookrightarrow \calp(n,r) \twoheadrightarrow \mathcal{Z}.
   \end{align}
where $\mathcal{Z}$ denotes $\Phi(\Gm)$. Note that $\uvtr$ can be decomposed into a direct sum of subspaces:  $\uvtr=\bigoplus_{l=0}^r\uvtr_l$ with $\uvtr_l=\sum_{\bi\in \caln^r_l}\bbc\eta_{\bi}$, $l=0,1,\cdots,r$ (see \S\ref{un rank} for $\caln^r_l$).
 So we can list a basis of the $(r+1)$-dimensional subalgebra $\Phi(\Gm)$ of $S(n+1,r)$ which consists of  $\epsilon_l$, $l=0,1,\ldots,r$. Here $\epsilon_l$ is defined via
\begin{align}\label{Z basis}
\epsilon_l|_{\uvtr_q}=\begin{cases} \id_{\uvtr_l}  ;& \text{ if } q=l,\cr
0; & \text{ otherwise. }
\end{cases}
\end{align}

Let us further investigate  the Levi subalgebra $\scrl(n,r)$.
 We consider the set
\begin{align*}
D:=\fsr.\{(\bp_l(n+1)^{r-l}), (\bq_l(n+1)^{r-l})\in \caln^r\times\caln^r
\mid \bp_l, \bq_l\in \un^l, l=0,1,\ldots,r\}.
\end{align*}
 By (S1)-(S3), the span of  $\xi_{\bi,\bj}$ with $(\bi,\bj)$ ranging over $D$ forms a subalgebra of $S(n+1,r)$, denoted by $\scrl(n,r)'$.

For a given $(g,c)\in \GL_n\times \Gm$ with $g=(g_{pq})_{n\times n}$, by the classical Schur algebra theory,
\begin{align}\label{class s g}
\Phi(\diag(g,c))|_{\vtr}=\sum_{(\bp,\bq)\in \un^r\times\un^r\slash\fsr}a_{\bp\bq}\xi_{\bp,\bq} \mbox{ with } a_{\bp\bq}=\prod_{i=1}^r g_{p_iq_i}.
\end{align}
  Recall the notation $I_{g}=\{(\bp,\bq)\in \un^r\times\un^r\slash\fsr\mid a_{\bp \bq}\ne 0\}$. Now let us return to the case of the enhanced tensor space. Then  $\Phi(\Diag(g,c))$ can be expressed as
\begin{align}\label{dot g}
 \Phi(\diag(g,c))=\sum_{(\bp,\bq)\in I_{g}} T^{g,c}_{\bp,\bq}
\end{align}
where $T^{g,c}_{\bp,\bq}=\sum_{{\overset{(\bi,\bj)\in D\slash \fsr}
{(\bi,\bj)\preccurlyeq (\bp,\bq)}}} a_{\bi\bj}\xi_{\bi,\bj}$ with \begin{align}\label{a coeff}
a_{\bi\bj}:=c^{r-t}\prod_{k=1}^t g_{p_kq_k}
\end{align}
 where $(\bi,\bj)\preccurlyeq (\bp,\bq)$ means
  \begin{align}\label{ bi and bj}
(\bi,\bj)=((p_{1}\ldots p_{t}(n+1)^{r-t}),(q_{1}\ldots q_{t}(n+1)^{r-t}))
\end{align}
with $(\bp,\bq)\sim ((p_1\ldots p_r), (q_1\ldots q_r))$ for some integer $t$ between $0$ and $r$. We set $\text{rank}_n(\bi):=t$ in \eqref{ bi and bj}.
By the above arguments, $\scrl(n,r)\subset \scrl(n,r)'$.

\begin{lemma}\label{gen G} The following statements hold.
\begin{itemize}
\item[(1)] The Levi  subalgebra  $\scrl(n,r)$ coincides with $\scrl(n,r)'$, this is to say, it  is spanned by $\xi_{\bi,\bj}$ with $(\bi,\bj)$ ranging over $D\slash \fsr$.

\item[(2)] Let $\scrl(n,r)_t$ be the subspace of $\scrl(n,r)$ spanned by $\xi_{\bi,\bj}$ with $\text{rank}_n(\bi)=t$. Then  $\scrl(n,r)=\sum_{t=0}^r\scrl(n,r)_t$ is a graded algebra.
\end{itemize}
\end{lemma}

\begin{proof} The second statement directly follows from the first one and the multiplication formula of $S(n,r)$. As to the first statement, it suffices to prove that for any $(\bi,\bj)\in D$, $\xi_{\bi,\bj}$   lies in $\scrl(n,r)$. 
Indeed, it is obvious that there exists $g\in\GL_n$ such that $(\bi,\bj)\in I_g$.  By applying Lemma \ref{general GL lem}, $\xi_{\bi,\bj}=\sum \bbc \Phi(h^{(i)})$ for some $h^{(i)}\in B(g)\subset \GL_n\hookrightarrow \GL_{n+1}$. This is desired.
\end{proof}

Thanks to Lemma \ref{general GL lem}, we have another result.
\begin{lemma}\label{diag gp} The space $T(n,r)$ 
spanned by  $\theta_{\bs,\bt}$ for all $(\bs,\bt)\in\Lambda$ is just the subalgebra generated by $\Phi(\Gm)$ and all $\theta_{({\sf{t}},v)}$ with $\sf{t}$ running through the subgroup $T(n)\subset \GL(n,\bbc)$ consisting of diagonal matrices, and $v$ ranging over $V$. 
Consequently, all $\theta_{\bs,\bt}$ are contained in $\calp(n,r)$.
\end{lemma}

\subsection{The structure of the parabolic Schur algebra $\calp(n,r)$}\label{sec: enh sch} Now it is a position to investigate the structure of $\calp(n,r)$.
For any given element in $\caln^r$, modulo the order we can write it in a standard form:
$$(1^{s_1}2^{s_2}\ldots n^{s_n}(n+1)^{r-|\bs|})$$
 where $\bs:=(s_1,\ldots,s_n)\in \bbn^n_r$. So for any basis element $\xi_{\bi,\bj}\in S(n+1,r)$ we can write
\begin{align}
\xi_{\bi,\bj}=\xi_{(1^{s_1}2^{s_2}\ldots n^{s_n}(n+1)^{r-|\bs|}), \bj}
\end{align}
with $\bj\in\caln^r$.
Set
\begin{align}\label{set E}
E:=\{(\tilde\pi_\bs, \bj): =&((1^{s_1}2^{s_2}\ldots n^{s_n}(n+1)^{r-|\bs|}),(j_1\ldots j_{|\bs|} (n+1)^{r-|\bs|}))\in\caln^r\times\caln^r\cr
&\mid\bs\in \bbn^n_r, (j_1\ldots j_{|\bs|})\in\caln^{|\bs|}\},
\end{align}
and $\widetilde E=\fsr.E$.

\begin{theorem}\label{enh sch stru}
Keep the above notations. The following statements hold.
\begin{itemize}
\item[(1)] The parabolic Schur algebra $\calp(n,r)$ is generated by $\Omega(n,r)$ and $\scrl(n,r)$ with product axioms as (S1)-(S3), this is to say, $\calp(n,r)$ is generated by $\Theta_\bs$ and $\xi_{\bi,\bj}$, with $\bs\in\bbn^n_r$ and $(\bi,\bj)\in D\slash\fsr$.

\item[(2)] The parabolic Schur algebra $\calp(n,r)$ has a basis  $\{\xi_{\tilde\pi_\bs,\bj}\mid (\tilde\pi_\bs,\bj)\in E\}$, and
$$\dim\calp(n,r)=\sum\limits_{k=0}^r{{n+k-1}\choose k}{{n+k}\choose k}.$$

\item[(3)] Set $$\calp(n,r)_i=\bbc\mbox{-span}\{\xi_{\tilde\pi_\bs,\bj}\in E\mid |\bs|\geq i\}\mbox{ for } i=0,1,\ldots,r,$$ and  $$\calp(n,r)_{r+1}=\bbc\mbox{-span}\{\xi_{\tilde\pi_\bs,\bj}\in E\mid |\bs|=r,\bj\in\caln^{r}\setminus \underline{n}^{r}\}. $$ Then  $\calp(n,r)_{i+1}$ is an ideal of $\calp(n,r)_{i}$, $i=0,\ldots,r$. Moreover, $\calp(n,r)_i/\calp(n,r)_{i+1}\cong S(n, i)\ltimes\mathfrak{a}_i$ as algebras, where $\mathfrak{a}_i$ is a nonzero abelian ideal for $i=0,\cdots, r-1$, and $\calp(n,r)_r/\calp(n,r)_{r+1}\cong S(n, r)$. 

\item[(4)] Furthermore,  $\calp(n,r)$ has another set of generators $\{\theta_{\bs,\bt},\xi_{\bi,\bj}\mid (\bs,\bt)\in\Lambda, (\bi,\bj)\in D\slash \fsr\}$.

\end{itemize}
\end{theorem}

\begin{proof}
The statement (1) follows from Lemmas \ref{gen V} and \ref{gen G}.

(2) Denote by $\calp(n,r)'$ the subspace of $S(n+1,r)$ spanned by    $\{\xi_{\tilde\pi_\bs,\bj}\mid (\tilde\pi_\bs,\bj)\in E\}$. By definition, $\calp(n,r)\subset \calp(n,r)'$. What remains is to show that for any ${\bf{z}}=\xi_{\tilde\pi_\bs,\bj}$ with $(\tilde\pi_\bs,\bj)\in E$, ${\bf{z}}$ must fall in $\calp(n,r)$. Modulo the order, we can write $(\tilde\pi_\bs,\bj)\in E$ as
$((p_1p_2\ldots p_s(n+1)^{r-s}), (q_1\ldots q_t(n+1)^{s-t}(n+1)^{r-s}))$ with $(p_1p_2\ldots p_s, q_1\ldots q_t)\in\un^s\times\un^t$ for $t\leq s$. Then
$${\bf{z}}=\xi_{p_1p_2\ldots p_s(n+1)^{r-s}, q_1\ldots q_t(n+1)^{r-t}}.$$
 Consider $${\bf{x}}:=\xi_{p_1p_2\ldots p_s(n+1)^{r-s}, q_1\ldots q_tp_{t+1}\ldots p_s (n+1)^{r-s}}\in \scrl(n,r)$$
and
 $${\bf{y}}:=\xi_{q_1\ldots q_tp_{t+1}\ldots p_s (n+1)^{r-s},q_1\ldots q_t(n+1)^{r-t}}\in \Omega(n,r).$$
Thanks to (S1), it follows that  ${\bf{xy}}=(c+1){\bf{z}}$ where {\small{
$$c=\#\{\sigma.(q_1,\ldots,q_t,p_{t+1},\cdots p_s)\mid \sigma\in (\frak{S}_s)^\bp_\bq\}$$}} for {\small{
$$(\frak{S}_s)^{\bp}_{\bq}:=\{\sigma\in\mathfrak{S}_s\mid \sigma.(p_1,\ldots,p_s)=(p_1,\ldots,p_s), \mbox{ while } \sigma.(q_1,\ldots,q_t,p_{t+1},\cdots p_s)\ne(q_1,\ldots,q_t,p_{t+1},\cdots p_s)\}.$$ }}
Consequently, by Lemmas \ref{diag gp} and \ref{gen G} we have that ${\bf{z}}$ really lies in $\calp(n,r)$. The dimension formula is obvious.

(3)  By (S1)-(S3), it is readily known that $\calp(n,r)_{i+1}$ is an ideal of $\calp(n,r)_{i}$, $i=0,\ldots,r-1$. And
$$\calp(n,r)_i/\calp(n,r)_{i+1}=\mathfrak{s}_i\oplus\mathfrak{a}_i$$
where
$$\mathfrak{s}_i=
\bbc\mbox{-span}\{\overline{\xi_{\tilde\pi_\bs,\bj}}\mid (\tilde\pi_\bs,\bj)\in E_i^{(1)}\},
\mathfrak{a}_i=\bbc\mbox{-span}\{\overline{\xi_{\tilde\pi_\bs,\bk}}\mid (\tilde\pi_\bs,\bk)\in E_i^{(2)}\}$$ with
{\small{
\begin{align*}
E_i^{(1)}=\{(\tilde\pi_\bs, \bj): =&((1^{s_1}2^{s_2}\ldots n^{s_n}(n+1)^{r-i}),\bj=(j_1\ldots j_{i} (n+1)^{r-i}))
\mid\bs\in \bbn^n, |\bs|=i, (j_1\ldots j_{i})\in\underline{n}^{i}\}.
\end{align*}}}
and {\small{
\begin{align*}
E_i^{(2)}=\{(\tilde\pi_\bs, \bk): =&((1^{s_1}2^{s_2}\ldots n^{s_n}(n+1)^{r-i}),(k_1\ldots k_{i} (n+1)^{r-i}))
\mid\bs\in \bbn^n, |\bs|=i, \bk=(k_1\ldots k_{i})\in\caln^{i}\setminus \underline{n}^{i}\}.
\end{align*}}}
The following mapping
\[
\begin{array}{rcl}
\Xi:\mathfrak{s}_i&\rightarrow&S(n, i)\vspace{3mm}\\
\overline{\xi_{\tilde\pi_\bs,\bj}}&\mapsto&\xi_{((1^{s_1}2^{s_2}\ldots n^{s_n}),(j_1\ldots j_{i}))}
\end{array}
\]
gives an algebra isomorphism between $\mathfrak{s}_i$ and $S(n,i)$. Furthermore, $\mathfrak{a}_i$ is an ideal of $\calp(n,r)_i/\calp(n,r)_{i+1}$.

(4) This statement follows from (1) and Lemma \ref{diag gp}.
\end{proof}

\section{Degenerate double Hecke algebras}
In this section, we introduce degenerate double Hecke algebras which will be important in the sequel arguments. For the symmetric group $\fsr$, we denote by $s_l=(l,l+1)$ for $l=1,\ldots,r-1$, the  transposition just interchanging $l$ and $l+1$, and fixing the others.

\subsection{Degenerate double Hecke algebras}
\subsubsection{}\label{sec: ddha}  For  given positive integers $r$ and $l$ with $r>l$, we consider the following algebra $\calh^l_r$ defined by generators $\{\bx_\sigma\mid \sigma\in\fsl\}\cup\{\bs_i\mid i=1,2,\cdots,r-1\}$ and relations as below.
\begin{align}\label{ddha 1}
\bs_i^2=1,\;\;\bs_i\bs_j=\bs_j\bs_i \mbox{ for } 0< i\ne j\leq r-1,\; |j-i|>1;
\end{align}
\begin{align}\label{ddha 2}
\bs_i\bs_j\bs_i=\bs_j\bs_i\bs_j \mbox{ for } 0< i\ne j\leq r-1,\;  |j-i|=1;
\end{align}
\begin{align}\label{ddha 3}
\bx_{\sigma}\bx_{\mu}= \bx_{\sigma\circ\mu} \mbox{ for } \sigma,\mu\in\fsl;
\end{align}
\begin{align}\label{ddha 4}
\bs_i\bx_{\sigma}=\bx_{s_i\circ\sigma},\;\; \bx_\sigma \bs_i=\bx_{\sigma\circ s_i} \mbox{ for }\sigma\in\fsl, i<l;
\end{align}
\begin{align}\label{ddha 5}
\bs_i\bx_{\sigma}=\bx_\sigma=\bx_\sigma \bs_i \mbox{ for }\sigma\in\fsl, i>l.
\end{align}
This is an infinite-dimensional associative algebra.  We call $\calh^l_r$  the $l$th degenerate double Hecke algebra of $\fsr$. By (\ref{ddha 3}), the subalgebra $X_l$ generated by $\bx_\sigma$ for $\sigma\in\fsl$ is isomorphic to $\bbc\fsl$.
As well as being a subalgebra of $\calh^l_r$, $\bbc\fsr$ is also a quotient, via the homomorphism $\calh^l_r\twoheadrightarrow \bbc\fsr$ mapping $\bs_i\mapsto s_i$ and $\bx_\sigma\mapsto 0$ for each $i=1,\ldots,r-1$ and $\sigma\in\fsl$.

 Additionally, we make an appointment that  $\calh^0_r:=\langle\fsr, X_0\rangle$ with $X_0=\bbc \bx_\emptyset$ satisfying $f=f.\bx_\emptyset=\bx_\emptyset.f$ for $f\in\calh^0_r$; and $\calh^r_r:=\langle\bbc \bs_i, \bx_\sigma\rangle$ with all $\bs_i, \bx_\sigma$, $i=1,\ldots,r-1$ and $\sigma\in\fsr$  satisfying (\ref{ddha 1})-(\ref{ddha 4}). Here and after, $\langle \Box \rangle$ stands for a $\bbc$-algebra generated by $\Box$.
 Then $\calh^0_r\cong \bbc\fsr$, and $\calh^r_r\cong \bbc \fsr^{[2]}$ with  $\fsr^{[2]}$ being a group giving rise to  a non-split extension $\fsr\hookrightarrow \fsr^{[2]}\twoheadrightarrow\fsr$.

\subsubsection{Full degenerate double Hecke algebras} Now we combine all $\calh^l_r$ ($l=0,1,\ldots,r$) into a full degenerate double Hecke algebras.
\begin{definition} The degenerate double Hecke algebra $\calh_r$ of $\fsr$ is an associative algebra  with generators $\bs_i$ ($i=1,\ldots,r-1$), and $\bx^{(l)}_\sigma$ for $\sigma\in\fsl$, $l=0,1,\ldots,r$, and with relations as (\ref{ddha 1})-(\ref{ddha 5}) in which $\bx_{\sigma},\bx_{\mu}$ are replaced by $\bx_{\sigma}^{(l)},\bx_{\mu}^{(l)}$, and additional ones:
\begin{align}\label{ddha 6}
\bx^{(l)}_\delta \bx^{(k)}_\gamma=0 \mbox{ for }\delta\in\fsl, \gamma\in\fsk \mbox{ with }k\ne l.
\end{align}
\end{definition}

 Naturally, as well as being a subalgebra of $\calh_r$, $\bbc\fsr$ is also a quotient of $\calh_r$, via the homomorphism $\calh_r\twoheadrightarrow \bbc\fsr$ mapping $\bs_i\mapsto s_i$ and $\bx^{(l)}_\sigma\mapsto 0$ for each $i=1,\ldots,r-1$ and $\sigma\in\fsl$, and $l=0,1,\ldots,r$.

\subsection{} Degenerate double Hecke algebras arise from the following question: 
\begin{question}
For the tensor representation  $\uvtr$ over the Levi-subgroup  $\GL(V)\times \Gm$, $\End_\bbc(\uvtr)^{\GL(V)\times \Gm}=?$
\end{question}

Let us begin the arguments with  turning  to $\uvtr$ which is regarded as a $\GL(V)$-module by fixing  $\eta$ and a $\Gm$-module by fixing $V$.

Keep the notations as before. In particular, we fix a basis $\{\eta_1,\ldots,\eta_n; \eta_{n+1}:=\eta\}$ for $\uV=V\oplus \bbc\eta$ with $V=\sum_{i=1}^n\bbc\eta_i$. Associated with this basis, $\GL(V)=\GL(n,\bbc)$, and $\GL(\uV)=\GL(n+1,\bbc)$. And $\GL(n,\bbc)$ is canonically regarded as a subgroup of $\GL(n+1,\bbc)$ by the established imbedding $\GL(V)\hookrightarrow \underline{\GL(V)} \hookrightarrow\GL(\uV)$ sending $g\in \GL(V)$ to $(g,0)\in \underline{\GL(V)}$ (see Example \ref{ex3} for the notations).

 In this view, we already know that the image of $\GL(V)\times\Gm$ by $\Phi$  is  $\scrl(n,r)$ (\S\ref{para sub}). In the following we will exactly determine  $\End_{\scrl(n,r)}(\uvtr)$ and then deduce the Levi version of  Schur-Weyl duality for  $\GL(n,\bbc)\times\Gm$.

\subsection{}\label{un rank} Recall that all $\eta_\bi=\eta_{i_1}\otimes \eta_{i_2}\otimes\cdots\otimes\eta_{i_r}$ for $\bi=(i_1,\ldots,i_r)\in\caln^r$ form a basis of $\uvtr$.
For a given $\bj=(j_1,\ldots,j_r)\in\caln^r$, there exists a unique $l\in\{0,1,\ldots,r\}$ and $\bj'_l=(j'_1,\ldots,j'_l)\in\un^l$ such that $\bj\sim (\bj'_l(n+1)^{r-l})$, $l$ is called the $\un$-rank of $\bj$, denoted by $\text{rk}_{\un}(\bj)$.
 All elements with $\un$-rank equal to $l$ constitute a subset of $\caln^r$, denoted by $\caln^r_l$. Clearly, $\caln^r=\bigcup_{l=0}^r \caln^r_l$.
  And $\uvtr$ is decomposed into a direct sum of subspaces:  $\uvtr=\bigoplus_{l=0}^r\uvtr_l$ for $\uvtr_l=\sum_{\bi\in \caln^r_l}\bbc\eta_{\bi}$, $l=0,1,\cdots,r$.

Clearly, each $\uvtr_l$ is stabilized under $\fsr$-action. Hence this action gives rise to a representation of $\fsr$ on $\uvtr_l$, denoted by $\Psi|_l$. For a subset $I=\{i_1,\ldots, i_l\}\subset \ur$ whose elements are assumed  to be ordered increasingly, we can write a subspace $\uvtr_I$ of $\uvtr_l$ as
$$\uvtr_I:=\bbc\text{-span}\{\eta_\bj=\eta_{j_1}\otimes\cdots\otimes\eta_{j_r}
\mid j_{i_k}\in\un, k=1,\ldots,l; j_d=n+1 \mbox{ for }d\ne i_k \}. $$
This means, for any $\eta_\bj\in \uvtr_I$, there exists
\begin{equation*}\tau_I=
\left( \begin{array}{cccc}
 1 & 2 & \cdots &r\cr
i_1 &i_2 &\cdots &i_r
\end{array}\right)\in\fsr
\end{equation*}
such that $\Psi(\tau_I)\eta_\bj=\eta_{\bj_l(n+1)^{r-l}}\in \uvtr_{\ul}$ for some $\bj_l\in\un^l$, where $\{i_{l+1},\cdots, i_n\}=\{j_d\mid d\in\underline{r}\setminus I\}$.  Conversely, for any $\eta_\bj\in \uvtr_l$, there exists $\tau\in\fsr$ and $\eta_{\bj'}\in \uvtr_{\ul}$ such that $\eta_\bj=\Psi(\tau)(\eta_{\bj'})$. What is more, as a vector space
$$\uvtr_l=\bigoplus_{I\subset \ur} \uvtr_I$$
where $I$ in the sum ranges over all subsets of $\ur$ consisting of  $l$ elements. In the above, each $\uvtr_I$ is a $\GL(V)$-module. The corresponding representation is a subrepresentation of $\GL(V)$ on $\uvtr_l$, the latter of which is also denoted by $\Phi$ for brevity.

For $I\subset \ur$ with $\# I=l$, denote by $\mbox{Sym}(I)$ the symmetric group of $I$ consisting all permutations of the $I$, which is isomorphic to $\fsl$. Naturally,  any $\sigma\in \mbox{Sym}(I)$ gives rise to a transformation on $\uvtr_I$ which just permutates the set $\{\eta_{\bj}\in\uvtr_I\}$ via changing the position of  factor $\eta_{j_i}$ into the position of $\eta_{j_{\sigma^{-1}(i)}}$ for all $i\in I$ and fixing the other factors. This gives rise to an representation of $\mbox{Sym}(I)$ on $\uvtr_I$. For $I=\ul$, the corresponding symmetric group is directly denoted by $\fsl$. The corresponding representation is denoted by $\Psi|_{\ul}$.

\subsection{Representations of $\calh_r$ on $\uvtr$}\label{comm ddha-g}
On $V^{\otimes l}$, there is a permutation representation $\Psi_l^V$ of $\fsl$ defined via $\Psi_l^V(\sigma)$ sending $v_1\otimes v_2 \otimes\cdots\otimes v_l$ onto $v_{\sigma^{-1}(1)}\otimes v_{\sigma^{-1}(2)}\otimes\cdots\otimes v_{\sigma^{-1}(l)}$ for $\sigma\in \fsl$. Keep in mind the notations $\ul=\{1,2,\ldots,l\}\subset\un$, and $\uvtr_{\ul}=V^{\otimes l}\otimes \eta^{\otimes r-l}$. Extending $\Psi_l^V$, we define the following linear operator on $\uvtr_{\ul}$
\begin{align}\label{sigma operator}
x_\sigma=\Psi^V_l(\sigma)\otimes \id^{\otimes r-l}\in\End_\bbc(\uvtr_{\ul})
\end{align}
for $\sigma\in\fsl$. Next we  extend $x_\sigma$ to an element $x_\sigma^{\ul}$ of $\End_\bbc(\uvtr_l)$ by annihilating any other summand $\uvtr_I$ with $I\ne \ul$.

Now let us look at the representation meaning of $\bs_i\in\calh_r^l$ in $\End_\bbc(\uvtr_l)$. This one keeps the role of $\Psi(s_i)$  such that the conjugation of $x_\sigma^{\ul}$ by $\Psi(s_i)$ will be an operator translating  $x_\sigma^{\ul}\in \End_\bbc(\uvtr_{\ul})$ to the forthcoming parallel one $x_\sigma^I\in \End_\bbc(\uvtr_I)$ for $I=\{1,2,\ldots,l-1,l+1\}$.

In general, for $\eta_\bj\in\uvtr_I$ with $I\subset\ur$ and $\#I=l$, we can write $\eta_\bj=\Psi(\tau_I)\eta_{\bj_l'(n+1)^{r-l}}$ for some $\bj_l'\in\un^l$.
Then $\Psi(\tau_I)\circ x_\sigma\circ\Psi(\tau_I^{-1})$  lies in $\End_\bbc(\uvtr_I)$ for any $\sigma\in\fsl$, which extends to an element of $\End_\bbc(\uvtr_l)$ by annihilating any other summand $\uvtr_J$ with $J\ne I$. This elements is denoted by $x_\sigma^I$. All $x_\sigma^I$ $(\sigma\in\fsl)$ generate a subalgebra in $\End_\bbc(\uvtr_l)$, denoted by $E_I$ which is isomorphic to $\bbc\Psi^V_l(\fsl)$. Set $x_e^l:=\sum_I x_e^I$, where $I$ in the sum ranges over all subsets of $\ur$ containing $l$ elements, and $e$ represents the identity element in $\fsl$. Then $x_e^l$ is just the identity mapping on $\uvtr_l$. Sometimes, $e$ also indicates  the identity element in $\fsr$ if the context is clear.



Now it is a position to demonstrate  a representation of the  degenerate double Hecke algebra $\calh_r$ on $\uvtr$.

\begin{lemma}\label{l rep of ddha}
The following statements hold.
\begin{itemize}

\item[(1)] For $1\leq l\leq r$, there is an algebra homomorphism $\Xi_l: \calh^l_r\rightarrow \End_\bbc(\uvtr_l)$ defined by sending $\bs_i\mapsto \Psi|_l(s_i)$ and $\bx_\sigma\mapsto x_\sigma^{\ul}$.


\item[(2)] For $l=0$,  there is an algebra homomorphism $\Xi_0: \calh^0_r\rightarrow \End_\bbc(\uvtr_l)$ defined by sending $\bs_i\mapsto \Psi|_l(s_i)=\id$.

\item[(3)] For any $l\in\{0,1,\ldots,r\}$,  $\Phi(\dot g)$ for any $g\in \GL(V)$ commutes with any elements from  $\Xi_l(\calh^l_r)$  in $\End_\bbc(\uvtr_l)$.

\item[(4)] Set  $E_l:=\bigoplus_{I\in\mathscr{S}_l} E_I$  with $\mathscr{S}_l:=\{I\subset \ur\mid \#I=l\}$, correspondingly $E_0:=E_\emptyset$. Then  $E_l$ is a  left $\Psi|_l(\fsr)$-module under the conjugation, i.e.  $\Psi|_l(\tau). x^I_\sigma=\Psi|_l(\tau) x^I_\sigma\Psi|_l(\tau)^{-1}$ for $\tau\in\fsr$. Furthermore,  $\Xi_l(\calh^l_r)=\Psi|_l(\bbc\fsr)E_l$.
\item[(5)] On the enhanced tensor space $\uvtr$, there is a representation $\Xi$ of $\calh_r$ defined via:
\begin{itemize}
\item[(5.1)] $\Xi|_{\bbc\fsr}=\Psi$;

\item[(5.2)] For any $\bx_\sigma\in \calh^l_r$, $l=0,1,\ldots,r$, $\Xi(\bx_\sigma)|_{\uvtr_l}=\Xi_l(\bx_\sigma)$ and $\Xi(\bx_\sigma)|_{\uvtr_k}=0$ for $k\ne l$.
\end{itemize}
\end{itemize}
\end{lemma}

\begin{proof}
(1) For  $l\leq r$, we need to show that $\Xi_l$ keeps the relations (\ref{ddha 1})-(\ref{ddha 5}).

Recall that for any $1\leq i\leq r-1$, $\bj\in\caln^r$, we have
\begin{align*}
\Xi_l(\bs_i)(\eta_\bj)&=\Psi|_l(s_i)(\eta_{\bj})\cr
&=\eta_{s_i(\bj)}\cr
&=\eta_{j_1}\otimes\cdots\otimes\eta_{j_{i+1}}\otimes\eta_{j_i}
\otimes\cdots\otimes\eta_{j_n}.
\end{align*}
Hence, it is readily known that
\begin{align}\label{ddha check1}
\Xi_l(\bs_i)^2=\id,\;\;\Xi_l(\bs_i)\Xi_l(\bs_j)=\Xi_l(\bs_j)\Xi_l(\bs_i) \mbox{ for } 0\leq i\ne j\leq r-1,\; |j-i|>1;
\end{align}
and
\begin{align}\label{ddha check2}
\Xi_l(\bs_i)\Xi_l(\bs_j)\Xi_l(\bs_i)=\Xi_l(\bs_j)\Xi_l(\bs_i)\Xi_l(\bs_j) \mbox{ for } 0\leq i\ne j\leq r-1,\;  |j-i|=1.
\end{align}

For any $\sigma,\mu\in\fsl$, and $\bj=(\bj_l(n+1)^{r-l})$ with $\bj_l\in\underline{n}^l$, we have
$$\Xi_l(\bx_{\sigma})\circ\Xi_l(\bx_{\mu})(\eta_{\bj})
=\eta_{\sigma\mu(\bj)}=\Xi_l(\bx_{\sigma\mu})(\eta_{\bj}).$$
And  $\Xi_l(\bx_{\sigma})\circ \Xi_l(\bx_{\mu})(\eta_{\bk})=0=\Xi_l(\bx_{\sigma\mu})(\eta_{\bk})$ for any $\bk\in\caln^r$ with $\eta_\bk\notin \uV_{\ul}^r$.
Hence,
\begin{align}\label{ddha check3}
\Xi_l(\bx_{\sigma})\circ\Xi_l(\bx_{\mu})=\Xi_l(\bx_{\sigma\mu})
=\Xi_l(\bx_{\sigma}\bx_{\mu}), \; \forall \sigma, \mu\in\fsl.
\end{align}

For any $\sigma\in\fsl, i<l$, and $\bj=(\bj_l(n+1)^{r-l})$ with $\bj_l=(j_1,\cdots, j_l)\in\underline{n}^l$, we have
\begin{align*}
\Xi_l(\bs_i)\circ \Xi_l(\bx_{\sigma})(\eta_{\bj})
&=\eta_{(j_{\sigma^{-1}(1)}\cdots j_{\sigma^{-1}(i+1)}j_{\sigma^{-1}(i)}\cdots j_{\sigma^{-1}(l)}(n+1)^{r-l})}\cr
&=\Xi_l(\bx_{s_i\circ\sigma})(\eta_{\bj}).
\end{align*}
And $\Xi_l(\bs_i)\circ \Xi_l(\bx_{\sigma})(\eta_{\bk})=0=\Xi_l(\bx_{s_i\circ\sigma})(\eta_{\bk})$ for  any $\bk\in\caln^r$ with $\eta_\bk\notin \uV_{\ul}^r$.
Hence,
\begin{align}\label{ddha check4}
\Xi_l(\bs_i)\circ \Xi_l(\bx_{\sigma})=\Xi_l(\bs_i\bx_{\sigma}), \sigma\in\fsl, i<l.
\end{align}
Moreover, similar arguments yield that
\begin{align}\label{ddha check5}
\Xi_l(\bx_{\sigma})\circ \Xi_l(\bs_i)=\Xi_l(\bx_{\sigma}\bs_i), \sigma\in\fsl, i<l.
\end{align}
and
\begin{align}\label{ddha check6}
\Xi_l(\bs_i)\circ\Xi_l(\bx_{\sigma})=\Xi_l(\bx_\sigma)=
\Xi_l(\bx_\sigma\bs_i) \mbox{ for }\sigma\in\fsl, i>l.
\end{align}
Now it follows from (\ref{ddha check1})-(\ref{ddha check6}) that $\Xi_l$ is an algebra homomorphism from $\calh^l_r$ to $\End_\bbc(\uvtr_l)$.


(2) In this situation, $\Xi_0$ obviously keeps the relations (\ref{ddha 1})-(\ref{ddha 5}). Hence, $\Xi_0$ is an algebra homomorphism from $\calh^0_r$ to $\End_\bbc(\uvtr_l)$.

(3) It suffices to show that for any $g\in \GL(V), 1\leq i\leq r-1, \sigma\in\fsl$,
$$\Phi(\dot g)\Xi_l(\bs_i)=\Xi_l(\bs_i)\Phi(\dot g)\,\,\text{and }\,\,\Phi(\dot g)\Xi_l(\bx_{\sigma})=\Xi_l(\bx_{\sigma})\Phi(\dot g).$$
Indeed,
\begin{align*}
\Phi(\dot g)\Xi_l(\bs_i)(\eta_{\bj})&=\Phi(\dot g)(\eta_{j_1}\otimes\cdots\otimes\eta_{j_{i+1}}\otimes\eta_{j_i}\cdots\otimes\eta_{j_r})\cr
&=\Phi(\dot g)(\eta_{j_1})\otimes\cdots\otimes\Phi(\dot g)(\eta_{j_{i+1}})\otimes\Phi(\dot g)(\eta_{j_{i}})\cdots\otimes\Phi(\dot g)(\eta_{j_{n}})\cr
&=\Xi_l(\bs_i)\Phi(\dot g)(\eta_{\bj}).
\end{align*}
Hence, $\Phi(\dot g)\Xi_l(\bs_i)=\Xi_l(\bs_i)\Phi(\dot g)$.

Furthermore, for any $\bj=(j_1,\cdots, j_n)=(\bj_l(n+1)^{r-l})$ with $\bj_l=(j_1,\cdots, j_l)\in\underline{n}^l$,
\begin{align*}
\Phi(\dot g)\Xi_l(\bx_{\sigma})(\eta_{\bj})&=\Phi(\dot g)(\eta_{j_{\sigma^{-1}(1)}}\otimes\cdots\otimes\eta_{j_{\sigma^{-1}(n)}})\cr
&=(\Phi(\dot g)\eta_{j_{\sigma^{-1}(1)}})\otimes\cdots\otimes(\Phi(\dot g)\eta_{j_{\sigma^{-1}(n)}})\cr
&=\Xi_l(\bx_{\sigma})\Phi(\dot g)(\eta_{\bj}).
\end{align*}
And $\Phi(\dot g)\Xi_l(\bx_{\sigma})(\eta_{\bk})=\Xi_l(\bx_{\sigma})\Phi(\dot g)(\eta_{\bk})=0$ for any $\bk\in\caln^r$ with $k_s=n+1$ for some $s\leq l$.
This implies that $\Phi(\dot g)\Xi_l(\bx_{\sigma})=\Xi_l(\bx_{\sigma})\Phi(\dot g)$, as desired.

(4) The first part follows from the second one.
 We only need to prove the latter by different steps.

 (i) First of all,  by a direct check, $\Psi|_l(\bbc\fsr)x^l_e$ is an associative  algebra
  because $\Psi|_l(\tau).x^l_e=\Psi|_l(\tau)x^l_e\Psi|_l(\tau^{-1})=x^l_e$ for $\tau\in\fsr$. Correspondingly, it is a left module of $\Psi|_l(\bbc\fsr)$.

 (ii) By the above, it is not hard to see that $\Xi_l(\calh_r^l)$ is spanned by $\Psi|_l(\tau)x^I_\sigma$ with  $\tau\in\fsr$, $\sigma\in\fsl$ and $I\in\mathscr{S}_l$. Hence $\Xi_l(\calh_r^l)=\Psi|_l(\bbc\fsr)E_l$.

 (iii) As to the algebra homomorphism, it can be directly verified.

(5) We need to show that $\Xi$ keeps the relations (\ref{ddha 1})-(\ref{ddha 5}). Since $\Psi|_{\fsr}$ is a representation of $\fsr$, $\Xi$ keeps the relations (\ref{ddha 1})-(\ref{ddha 2}). Moreover, note that $\Xi|_{_{\uvtr_l}}=\Xi_l$ for $0\leq l\leq r$, we have
\begin{align}\label{hecke rep  3}
\Xi(\bx_{\sigma}^l)\circ\Xi(\bx_{\mu}^k)= \Xi_l(\bx_{\sigma}^l)\Xi_k(\bx_{\mu}^k)=\delta_{kl}\Xi_l(\bx_{\sigma\circ\mu}^k)
=\delta_{kl}\Xi(\bx_{\sigma}^l\bx_\mu^k) \mbox{ for } \sigma,\mu\in\fsl, k,l\in\ur.
\end{align}
\begin{align}\label{hecke rep  4}
\Xi(\bs_i)\circ\Xi(\bx_{\sigma}^l)=\Psi(s_i)\Xi_l(\bx_{\sigma}^l)
=\Xi_l(\bx_{\bs_i\circ\sigma}^l)=\Xi(\bs_i\bx_{\sigma}^l), \sigma\in\fsl, i<l.
\end{align}

\begin{align}\label{hecke rep  5}
\Xi(\bx_{\sigma}^l)\circ\Xi(\bs_i)=\Xi_l(\bx_{\sigma}^l)\Psi(s_i)=\Xi_l(\bx_{\sigma\circ s_i}^l)=\Xi(\bx_{\sigma}^l\bs_i), \sigma\in\fsl, i<l.
\end{align}

\begin{align}\label{hecke rep  6}
\Xi(\bs_i)\circ\Xi(\bx_{\sigma}^l)=\Psi(s_i)\Xi_l(\bx_{\sigma}^l)
=\Xi_l(\bx_{\sigma}^l)=\Xi(\bs_i\bx_{\sigma}^l) \mbox{ for }\sigma\in\fsl, i>l.
\end{align}

\begin{align}\label{hecke rep  7}
\Xi(\bx_\sigma^l)\circ\Xi(\bs_i)=\Xi_l(\bx_{\sigma}^l)\Psi(s_i)
=\Xi_l(\bx_{\sigma}^l)=\Xi(\bx_{\sigma}^l\bs_i)\mbox{ for }\sigma\in\fsl, i>l.
\end{align}
So $\Xi$ is an algebra homomorphism, thereby  a representation of $\calh_r$.
\end{proof}

We will further investigate  this representation in the next section.

\subsection{Finite dimensional DDHAs} Keep the notations as above.

\begin{definition}
 Set $D(n,r):=\Xi(\calh_r)$, which we call the finite-dimensional degenerate double Hecke algebra of $\fsr$ (f.d. DHHA for short).
 \end{definition}

 Set $D(n,r)_l:=\Xi_l(\calh^l_r)$ for $l=0,1,\ldots, r$. We can extend the action of $D(n,r)_l$ on the whole of $\uvtr$ as follows.  Set $$\Psi_l(\sigma):=\Psi|_{\uvtr_l}(\sigma)\circ x^l_e.$$
  Then $\Psi_l$ defines a representation of $\fsr$ on $\uvtr_l$. Now $\Psi_l$ extends a representation of $\fsr$ on $\uvtr$ by letting $\Psi_l(\fsr)(\eta_{\bj})=0$ for $\eta_{\bj}\in\uvtr_k$ with $k\neq l$. Similarly, each $x_{\sigma}^I$ extends to an element in $\End_\bbc(\uvtr)$ with trivial action on $\uvtr_k$ for $k\ne\# I$.  So $D(n,r)_l$ annihilates $\uvtr_k$ for $k\ne l$.
 Thus, each $D(n,r)_l$ becomes a subalgebra of $\cala$, and  two different  such subalgebras are commutative.

\begin{lemma}\label{sum lem}
Keep the notations as above. Then the following statements hold.
\begin{itemize}
\item[(1)] The f.d. DDHA $D(n,r)$ is a direct sum of all $D(n,r)_l$ \rm($\l=0,\cdots, r$\rm), i.e.,
\begin{equation}\label{sum decom}
 D(n,r)=\bigoplus_{l=0}^r D(n,r)_l.
\end{equation}
\item[(2)] Set $d(n,l):=\sum_{\lambda\in P(l,n)} (\dim S_l^{\lambda})^2$ with  $S_l^{\lambda}$ denoting the irreducible Specht module of $\mathfrak{S}_l$ corresponding to $\lambda\in \text{Par}(l,n)$ (see \S\ref{sec: 4.1} for the notation). Then
    $\dim D(n, r)_l=d(n,l){r\choose l}^2$ and
$$\dim D(n,r)=\sum\limits_{l=0}^rd(n,l){r\choose l}^2.$$
\item[(3)] For $l\in\{0,1,\ldots,r\}$ there is a basis of $E_l$: $\{ x^I_{\sigma_{l,i}}\mid I\in\mathscr{S}_l, \sigma_{l,i}\in\fsl, i=1,\ldots,d(n,l)\}$ with $\sigma_{l,1}=\id$, and ${r\choose l}^2$ elements $c_{J,I}\in \Psi|_l(\fsr)$ such that
 $D(n,r)_l$ has a basis $\{c_{J,I}x_{\sigma_{l,i}}^I\mid (I,J)\in \mathscr{S}^2_l, i=1,\ldots,d(n,l)\}$.  In particular, $x_e=\sum_{l=1}^r\sum_{I\in\mathscr{S}_l}x^I_e + x_\emptyset$ is just the identity of $D(n,r)$.

\item[(4)] The following  decomposition of $D(n,r)$ into a direct sum of subspaces holds:
    $$D(n,r)=\bbc\Psi(\fsr)x_e\oplus \bigoplus_{l=1}^r\bigoplus_{\overset{(J,I)\in\mathscr{S}^2_l}{ i=1,\ldots,d(n,l)}}\bbc c_{J,I} x^I_{\sigma_{l,i}}.$$
\end{itemize}
\end{lemma}

\begin{proof}
(1) By definition, $\bigoplus_{l=0}^r D(n,r)_l$ is really a direct sum in $\End_\bbc(\uvtr)$. In the following, we prove (\ref{sum decom}).
 Let $\tau\in\fsr$. Then  for any $\bi\in\caln^r$, we have
$$\Psi(\tau)(\eta_{\bi})=\Psi|_{s}(\tau)(\eta_{\bi})=
\sum\limits_{l=0}^r\sum\limits_{\stackrel{I\subset\un}{\# I=l}}\Psi|_l(\tau)x_e^I(\eta_{\bi}).$$
where $s=\text{rk}_{\un}(\bi)$.
This implies that
$$\Psi(\tau)=\sum\limits_{l=0}^r\sum\limits_{\stackrel{I\subset\ur}{\# I=l}}\Psi|_l(\tau)x_e^I.$$
Hence,
$$ D(n,r)\subset\bigoplus_{l=0}^r D(n,r)_l.$$
On the other hand, since
$$\Psi|_l(\sigma)(\eta_{\bi})=\delta_{l,\text{rk}_{\un}(\bi)}\Psi(\sigma)(\eta_{\bi})=\Psi(\sigma)x_e^l(\eta_{\bi}),\forall\,\bi\in\caln^r,$$
it follows that
$$ \Psi|_l(\sigma)=\Psi(\sigma)x_e^l.$$
This implies that
$$ \bigoplus_{l=0}^r D(n,r)_l\subset D(n,r).$$
We complete the proof for (\ref{sum decom}).

(2) Due to (1), it suffices to prove the formula $\dim D(n, r)_l=d(n,l){r\choose l}^2$.   Firstly, for any given $I\in\mathscr{L}_l$,  the subalgebra $D_I$ generated by all $x^I_\sigma$ ($\sigma\in\fsl$) is isomorphic to $\bbc\Psi_l^V(\fsl)$ which is of dimension equal to $d(n,l)$ by the  representation theory of symmetric groups. So the dimensions of all those $D_I$ are the same. Secondly,  by Lemma \ref{l rep of ddha}(5) $D(n,r)_l$ is spanned by $\Psi(\tau)x^I_\sigma$ for all  $\tau\in\fsr, \sigma\in\fsl$ and $I\subset\ur$ with $\#I=l$. On the other hand, for  any subsets $I=\{i_1<i_2<\cdots<i_l\}, J=\{j_1<j_2<\cdots<j_l\}$ of $\ur$ with $\# I=\# J=l$, take $\tau_{IJ}\in\fsr$ such that $\tau_{IJ}(j_k)=i_k$ for $k=1,\cdots, l$.
 For any two $\tau_1$ and $\tau_2\in\fsr$, if $\tau_1^{-1}\circ\tau_2$ stabilizes $I$, then $\Psi_l(\tau_1)x_\sigma^I=\Psi_l(\tau_2)x_{\sigma'}^I$ for some $\sigma'\in\fsl$. Actually,  for $\tau\in \fsr$ satisfying $\tau|_I\in\text{Sym}(I)$ which is isomorphic to $\fsl$, say by $\theta$,  we have $\Psi_l(\tau)x_{\sigma}^I=x_{\theta(\tau|_I)\circ\sigma}^I$ for any $\sigma\in\fsl, I\subset\ur$.

By the above arguments, we have that for any ordered pair $(I,J)$, the above $\tau_{IJ}$ changes a set of basis of $D_J$ into the ones of $D_I$. There are ${r\choose l}^2$  such ordered pairs $(I,J)$. Summing up,  along with  the definition of $D_I$ we have $\dim D(n,r)_l=d(n,l){r\choose l}^2$, as desired.

(3) This statement follows from the arguments in (2).

(4) Note that the chosen basis elements for every $E_I$ in the statement (3) include $x^I_e$. We make a proper adjustment of the basis such that the new basis contains $x_e$. Then the corresponding decomposition follows.
\end{proof}

\begin{remark} \label{remark: 3.6} If $n>l$,
then  the following set
\begin{equation*}\label{basis}
 \{\Psi_l(\tau_{JI})x_{\sigma}^I\mid I\subset\ur, J\subset\ur, \# I=\# J=l,\sigma\in\fsl\}
\end{equation*}
is a basis of $D(n, r)_l$. Moreover, the algebra structure of $D(n, r)_l$ is given by
\begin{align}\label{algebra structure}
(\Psi_l(\tau_{LK})x_{\mu}^K)(\Psi_l(\tau_{JI})x_{\sigma}^I)=\begin{cases} \Psi_l(\tau_{LI})x_{\mu\circ\sigma}^I, &\mbox{ if }J=K;\cr
0, &\mbox{ otherwise.}
\end{cases}
\end{align}
\end{remark}
 {\sl{Proof of Remark \ref{remark: 3.6}:}}
  Note that in this case, the elements $x_\sigma^I$ for all $\sigma\in\fsl$ and for any given $I$ in the remark are linearly independent in $D_I$.
  The first statement  follows from the arguments in the proof of Lemma \ref{sum lem}(3).
  To show (\ref{algebra structure}), we can assume that $I=\ul$ without loss of generality. We note that the non-vanishing range of the operator $\Psi_l(\tau_{JI})x_{\sigma}^I$ is contained in $\uvtr_J$ for $J\neq K$. Hence, $(\Psi_l(\tau_{LK})x_{\mu}^K)(\Psi_l(\tau_{JI})x_{\sigma}^I)=0$. For the case $J=K$, we have
\begin{align*}
(\Psi_l(\tau_{LJ})x_{\mu}^J)(\Psi_l(\tau_{JI})x_{\sigma}^I)(\eta_{\bi})=\begin{cases} \Psi_l(\tau_{LI})\eta_{(\mu\circ\sigma)\bi}, &\mbox{ if }\bi\in\caln^r_l, i_k\in\un,\,\forall\,k\in I;\cr
0, &\mbox{ otherwise.}
\end{cases}
\end{align*}
Hence, (\ref{algebra structure}) follows.





\section{Duality related to the degenerate double Hecke algebra and the branching duality formula}

In this section, with aid of some structural property of Levi and parabolic Schur algebras, we establish a duality between the degenerate double Hecke algebra $\calh_r$ and $\GL(n,\bbc)\times\Gm$ on the tensor space $(\bbc^{n+1})^{\otimes r}$. We will keep the nations as before.

\subsection{} \label{sec: 4.1} Let us first recall the classical branching law for general linear groups. For positive integers $m,r$, set $$\text{Par}(r,m)=\{\mu=(\mu_1,\cdots,\mu_m)\in\mathbb{N}^n\mid |\mu|=r, \mu_1\geq\mu_2\geq\cdots\geq\mu_m\}. $$
 An element  from $\text{Par}(r,m)$ is usually  called a dominant weight, which is actually a partition of $r$ into $m$ parts (zero parts are allowed). For another weight $\lambda \in\text{Par}(l,n)$ with $l\leq r$ and $n\leq m$, we call $\lambda$ interlaces $\mu$ if
 $\mu_1\geq \lambda_1\geq \mu_2\geq \lambda_2\geq\cdots\geq\mu_{m-1}\geq\lambda_{m-1}\geq \mu_m$
 where $\lambda_{n+1}=\cdots=\lambda_{m-1}=0$. We denote $\lambda\lesssim\mu$ if $\lambda$ interlaces $\mu$.

 For a given $\mu\in\text{Par}(r,n+1)$, let  $L^\mu_{n+1}$ be an irreducible $\GL(n+1,\bbc)$-module with highest weight $\mu$. Then as a module over $\GL(n,\bbc)$, there is a unique decomposition
 $$ L^\mu_{n+1}=\bigoplus_{l=0}^{r}\bigoplus_{\lambda\in \text{Par}(l,n)} \delta_{\lambda\lesssim\mu}L^\lambda_n$$
with
\begin{align*}
\delta_{\lambda\lesssim\mu}=\begin{cases} 1, &\mbox{ if }\lambda\lesssim\mu;\cr
0, &\mbox{ otherwise.}
\end{cases}
\end{align*}
There is a similar branching law for symmetric groups (see (\ref{sym gp branch rule})).

\subsection{}
In the following, we  give the decomposition of $\uvtr_l$ as a $\calh^l_r$-module into a direct sum of irreducible modules. For this, denote by $L^{\lambda}_n$ the irreducible $\GL(n,\bbc)$-module with highest weight $\lambda$  for $\lambda\in \text{Par}(l,n)$, and denote by $S^{\lambda}_l$ the irreducible Specht module over $\mathfrak{S}_l$ corresponding to $\lambda$.

Recall that $\GL(n,\bbc)$ has a standard Borel subgroup consisting of all upper triangular invertible matrices, denoted by $\frak{B}$. Denote by $\frak{N}$ the unipotent radical of $\frak{B}$ which consists of all unipotent upper triangular matrices.

\begin{lemma}\label{ddha decomp l} Keep the notations as above. As a $\calh^l_r$-module, $\uvtr_l$ decomposes into a direct sum of irreducible modules  as follows
\begin{align}\label{decomp l level}
\bigoplus_{\lambda\in \text{Par}(l,n)}(D^\lambda_l)^{\oplus \dim L^\lambda_n}
\end{align}
where as a $\bbc\fsl$-module,
$$D^\lambda_l\cong\underbrace{S^{\lambda}_l\oplus S^{\lambda}_l\oplus\cdots\oplus S^{\lambda}_l.}_{{r\choose l}\,\,\text{\rm times}}$$
\end{lemma}

\begin{proof} For any given $J=\{j_1,j_2,\ldots,j_l\}\subset \underline{r}$, we have a subspace $\uvtr_J$ in $\uvtr_l$, which is spanned by $\eta_{\bi}=\eta_{i_1}\otimes\cdots\otimes\eta_{i_r}$ with $i_{j_1},\ldots,i_{j_l}\in\un$. Then
\begin{align}\label{Level l decomp}
\uvtr_l=\bigoplus_{\overset{J\subset\underline{r}}{\#J=l}}\uvtr_J
\end{align}
where in the sum, $J$ ranges over all subsets of $\underline{r}$ consisting of $l$-elements.
The number of such $J$ is exactly $r\choose l$. For a fixed $J$, $\uvtr_J$ admits an $\fsl$-action, which just permutates the position indicated by $J$, this is to say, for any $\sigma_l\in\fsl$,
$\sigma_l.\eta_\bi= \eta_{\sigma_l.\bi}$ if defining $\sigma_l.\bi=\delta.((\sigma_l.(\bi'_l))(n+1)^{r-l})$ for $\bi=\delta.(\bi'_l(n+1)^{r-l})$.
On the other hand, $\uvtr_J$ becomes a $\GL(V)$-module with every factor $\eta$ fixed.

Thus $\uvtr_J$ can be regarded as an $r$-tensor space of $V$ with $(\GL_n\times \fsl)$-action. Thanks to the classical Schur-Weyl duality, as a $(\GL_n\times \fsl)$-module we have the following decomposition
\begin{align}\label{J part SW}
\uvtr_J\cong \bigoplus_{\lambda\in \text{Par}(l,n)}L^{\lambda}_n\otimes S^{\lambda}_l.
\end{align}
Let us show the meaning of  $S^{\lambda}_l$ in the above decomposition by demonstrating the standard representatives in the isomorphism class of $S^\lambda_l$. For  the simplicity of arguments, we might as well suppose $J=\ul$ which is equal to $\{1,2,\ldots,l\}$ without any loss of generality. Take $\lambda=(\lambda_1,\lambda_2,\ldots,\lambda_n)\in \text{Par}(l,n)$.
Set
\begin{align*}
\eta(\lambda)=&\underbrace{\eta_1\otimes\cdots\otimes\eta_1}_{\lambda_1 \text{ factors}}\otimes \underbrace{\eta_2\otimes\cdots\otimes\eta_2}_{\lambda_2 \text{ factors}}\otimes\cdots\otimes\underbrace{\eta_n\otimes\cdots\otimes\eta_n}_{\lambda_n \text{ factors}}
\otimes\eta\otimes\cdots\otimes\eta.
\end{align*}
Consider the orbit $\fsl.\eta(\lambda)$. The space spanned by this  orbit forms an $\fsl$-module, which is denoted by $I^\lambda_{\ul}$. Then $I^\lambda_{\ul}$ is actually the $\lambda$-weighted space of $\GL(n,\bbc)$-module $\uvtr_{\ul} $. In particular, $I^\lambda_{\ul}$ contains all maximal $\lambda$-weighted vectors in $\GL(n,\bbc)$-module $\uvtr_{\ul}$, all of which by definition, are $\frakN$-invariants in $I^\lambda_{\ul}$. Furthermore,  the invariant space in $I^\lambda_{\ul}$ under the action of $\frakN$ is exactly the unique direct summand isomorphic to $S^\lambda_l$, in the complete reducible decomposition of $I^\lambda_{\ul}$ as $\fsl$-module (see \cite[\S9.1.1]{GW}). This summand is denoted by $S^\lambda_{\ul}$. More generally, each $\uvtr_I$ for $I\subset \underline{r}$ (with $\#I=l$) admits the corresponding $I^\lambda_I$ which contains  a unique direct summand $S^\lambda_I$ paralleling to $S^\lambda_{\ul}$. This means, there exists $\sigma\in \fsr$ such that the conjugation by $\Psi|_l(\sigma)$ maps $S^\lambda_{\ul}$ onto $S^\lambda_I$.
Obviously,
$$D^\lambda_l:=\bigoplus_{\overset{I\subset \underline{r}}{\#I=l}}S^\lambda_I$$
 is a $\Psi|_l(\fsr)$-module, which is as an $\fsl$-module, isomorphic to the direct sum of $r\choose l$ copies of $S^\lambda_l$.

  Next we prove that $D^\lambda_l$ is an irreducible module over $\calh^l_r$. For this, we only need to show that for any given nonzero vector $w\in D^\lambda_l$, the cyclic $\calh^l_r$-submodule $W$ generated by $w$ must coincide with $D^\lambda_l$ itself. Now, we  write $w=\sum_{\text{some }I} w_I$ with some nonzero $w_I\in\uvtr_I$.
 Fix one $w_I$. Recall  $\calh^l_r$ has a subalgebra $E_I$ by abuse of the notation (see \S\ref{comm ddha-g}), which is isomorphic to $\bbc\Psi_l^V(\fsl)$. So the cyclic module $E_Iw_I$ of $E_I$ coincides with the irreducible $\bbc\fsl$-module $S^\lambda_I$. This implies that there is some $x^I_\sigma$ for $\sigma\in\fsl$ such that $x^I_\sigma w_I\ne 0$. On the other hand, by definition (see \S\ref{comm ddha-g}) $x^I_\sigma\sum_{J\ne I} w_J=0$. Therefore, $x^I_\sigma w=x^I_\sigma w_I$ is a nonzero vector in $S^\lambda_I$. Hence the submodule $W$ contains an irreducible $E_I$-module $S^\lambda_I$. By $\bbc\fsr$-action, we finally have that the $\calh^l_r$-submodule $W$ coincides with $\bigoplus_{\overset{J\subset \underline{r}}{\#J=l}}S^\lambda_J$, which is equal to $D^\lambda_l$.

 According to the previous analysis,  $D^\lambda_l$ is really an irreducible $\calh^l_r$-module. From the above arguments, along with (\ref{Level l decomp}) and (\ref{J part SW}), the lemma follows.
\end{proof}

\subsection{}\label{Phi Upsilon}
Let us first notice that $\uvtr$ naturally becomes a representation space of $\cale(n,r)$ (and of $S(n+1,r)$ more generally) by defining
\begin{align}\label{Ups rep}
\xi(\eta_\bi)=\sum_{\bk\in\caln^r}\xi(c_{\bk,\bi})\eta_\bk
\end{align}
 for any $\xi\in \cale(n,r)$, and any basis elements $\eta_\bi=\eta_{i_1}\otimes\cdots\otimes \eta_{i_r}\in $ of $\uvtr$ ($\bi\in \caln^r$).  This representation is denoted by $\Upsilon$. We are actually considering $\Upsilon:\cale(n,r)\rightarrow \End_\bbc(\uvtr)$.

As before, let $\Phi:\GL(n+1)\rightarrow \cala:=\End_\bbc(\uvtr)$ be the natural representation. By the classical Schur algebra theory, we can identify the image of $\Phi$  with $S(n+1,r)$, in aid of $\Upsilon$.

Consequently, we have the following corollary to Lemma \ref{ddha decomp l}.
\begin{corollary}\label{ddha decomp full}  Keep the notations as above. As a $\calh_r$-module, $\uvtr$ decomposes into a direct sum of irreducible modules  as follows
\begin{align}\label{decomp l level}
\bigoplus_{l=0}^r\bigoplus_{\lambda\in \text{Par}(l,n)}(D^\lambda_l)^{\oplus \dim L^\lambda_l}.
\end{align}
\end{corollary}

\subsection{} We have the following Levi version of  Schur-Weyl duality from  $\GL(n+1,\bbc)$ to $\GL(n,\bbc)\times \Gm$.

 \begin{theorem}\label{res swd} (Levi Schur-Weyl duality)
  Keep the above notations. The following statements hold.
 \begin{itemize}
 \item[(1)] $\End_{D(n,r)}(\uvtr)=\bbc\Phi(\GL(V)\times\Gm)$.
 \item[(2)] The following duality holds:
   \begin{align}\label{Res SW}
  \End_{\scrl(n,r)}(\uvtr)&=D(n,r);\cr
 \End_{D(n,r)}(\uvtr)&=\scrl(n,r).
  \end{align}
\end{itemize}
\end{theorem}

\begin{proof} (1) Recall that $\GL(V)\times \Gm$ can be regarded as a closed subgroup of $\GL(\uV)$ (see \S\ref{para sub}).
By Lemma \ref{l rep of ddha}(3) (along with some direct computation), $\bbc\Phi(\GL(V)\times \Gm )\subset \End_{D(n,r)}(\uvtr)$. We need to show the opposite inclusion.

Note that $D(n,r)=\Xi(\calh_r)$. So $\End_{D(n,r)}(\uvtr)\subset \End_{\bbc\fsr}(\uvtr)=\bbc\Phi(\GL(\uV))$. The last equation is due to the classical Schur-Weyl duality for $\GL(\uV)$ and $\fsr$. Thus, for any $\phi\in \End_{D(n,r)}(\uvtr)$, we can write $\phi=\sum_{h\in \GL(\uV)}a_h\Phi(h)$ for $a_h\in\bbc$.
In the following, we will show that $\phi$ must lie in $\bbc\Phi(\GL(V)\times\Gm)$.

We identify $\GL(\uV)$ with $\GL(n+1,\bbc)$, and $\Phi(\GL(\uV))$ with $S(n+1,r)$ by the classical theory of  Schur algebras. Thus, we rewrite $\phi=\sum_{(\bi,\bj)\in\caln^r\times\caln^r\slash\fsr} a_{\bi\bj}\xi_{\bi,\bj}$ for $a_{\bi\bj}\in \bbc$ satisfying $a_{\bi\bj}=a_{\tau.\bi,\tau.\bj}$ for any $\tau\in\fsr$. Furthermore,  $\phi=\sum_{l=0}^r \phi_l$ with $$\phi_l=\sum_{\overset{(\bi,\bj)\in\caln^r\times\caln^r\slash\sim}{\text{rk}_{\un}\bj=l}}
a_{\bi\bj}\xi_{\bi,\bj}.$$
In this view, by Lemma \ref{gen G} we only need to prove that all $\phi_l$ lie in the Levi subalgebra $\scrl(n,r)$. The arguments proceed in steps.

(1-1) By the assumption $\phi\in \End_{D(n,r)}(\uvtr)$, $\phi\circ x^{\ul}_\sigma=x^{\ul}_\sigma\circ \phi$ for all $\sigma\in\fsl$, and $l=0,1,\ldots,r$.
 We first claim that $\phi$ stabilise $\uvtr_l$ for every $l$.

Actually, if not so, then there must be a basis element $\eta_\bj\in \uvtr_l$ not obeying the stability  claimed above for $\phi$.  From (\ref{Ups rep}) it follows $\phi(\uvtr_l)=\phi_l(\uvtr_l)$.  Violation of the stability implies that $\phi_l(\eta_\bj)=\eta_l+\eta_k$ with $\eta_l\in\uvtr_l$ but $\eta_k\notin\uvtr_l$. Suppose $\bj=\tau(\bj'_l(n+1)^{r-l})$ for $\tau\in\fsr$, and $\bj'_l\in \un^l$. Then $(\tau x^{\ul}_{\id}\tau^{-1})(\eta_\bj)=\eta_\bj$. On the other side, $(\tau x^{\ul}_{\id}\tau^{-1})\circ \phi (\eta_\bj)=\phi\circ (\tau x^{\ul}_{\id}\tau^{-1})(\eta_\bj)$. This leads to a contradiction that $\tau x^{\ul}_{\id}\tau^{-1}(\eta_l+\eta_k)=\eta_l+\eta_k$. So the claim is true.

(1-2) By (1-1), $\phi_l$ is commutative with $x^{\ul}_\sigma$ for all $\sigma\in\fsl$. From this, by the same arguments as in (1-1) it follows that $\phi_l$ stabilises $\uvtr_{\ul}$.  Hence $\phi_l$ is actually commutative with $\fsl$ in $\End_\bbc(\uvtr_{\ul})$.

Recall that $\uvtr_{\ul}=V^{\otimes l}\times \eta^{\otimes r-l}$, and $x^{\ul}_\sigma=\Psi_l^V(\sigma)\otimes \id^{\otimes l-r}$. By applying the classical Schur Weyl duality, we have that $\phi|_{\uvtr_{\ul}}\in \Phi(\GL(V)|_{\uvtr_{\ul}}=\Phi(\GL(V)\times\Gm)|_{\uvtr_{\ul}}$. Note that $\phi|_{\uvtr_{\ul}}=\phi_l|_{\uvtr_{\ul}}$. From this, it is deduced that $\phi_l|_{\uvtr_{\ul}}\in \Phi(\GL(V)\times\Gm)|_{\uvtr_{\ul}}$.  By definition, we can further deduce that $\phi_l|_{\uvtr_{l}}\in \Phi(\GL(V)\times \Gm)|_{\uvtr_l}$.  Keep in mind  $\phi_l|_{\uvtr_k}=0$ for $k\ne l$  in the same sense as in \S\ref{Phi Upsilon}. By definition and Lemma \ref{gen G}(2), it follows that $\phi_l\in \scrl(n,r)_l\subset\scrl(n,r)$.

By the analysis in the beginning, we accomplish the proof of the first statement.

(2) The second equation in (\ref{Res SW}) follows from (1). Note that 
$\uvtr$ is completely reducible over $D(n,r)$ (Corollary \ref{ddha decomp full}). The first equation in (\ref{Res SW}) follows from the classical double commutant theorem (see \cite[\S4.1.5]{GW}).
\end{proof}

 As a consequence of Theorem \ref{res swd}, we have the following result on classification of irreducible $D(n, r)$-modules.

\begin{corollary} The set consisting of $D^\lambda_l$ for $\lambda\in \text{Par}(l,n)$ and $l=0,1,\ldots,r$ form a representative set of the isomorphism class of irreducible $D(n, r)$-modules.
\end{corollary}

\begin{proof}
Since $\scrl(n,r)$ is semisimple, so is $D(n,r)$ by Theorem \ref{res swd} and the Double Commutant  Theorem. Furthermore, each $D(n,r)_l$ is semisimple for $l=0,1,\ldots,r$ by Lemma \ref{sum lem}(1). Thanks to Lemma \ref{ddha decomp l}, $D^\lambda_l$ is an irreducible $D(n,r)_l$-module for any $\lambda\in \text{Par}(l,n)$, and $D^\lambda_l\not\cong D^\mu_l$ for $\lambda\neq\mu$.  Moreover,
\begin{align*}
\sum\limits_{\lambda\in\text{Par}(l,n)}(\dim D^\lambda_l)^2&=\sum\limits_{\lambda\in\text{Par}(l,n)}{r\choose l}^2(\dim S_l^{\lambda})^2\cr
&={r\choose l}^2\sum\limits_{\lambda\in\text{Par}(l,n)}(\dim S_l^{\lambda})^2\cr
&={r\choose l}^2 d(n,l)
\cr
&=\dim D(n,r)_l    \;\;\;\; (\mbox{Lemma \ref{sum lem}(4)}).
\end{align*}
Hence, the set $\{D^\lambda_l\mid \lambda\in \text{Par}(l,n)\}$ exhausts all representatives of isomorphism classes of irreducible  $D(n,r)_l$-modules. Consequently, the desired assertion holds.
\end{proof}

\subsection{Branching duality formula}\label{interlaces} From Theorem \ref{res swd}
 we derive a dimension relation between  irreducible modules from $(\GL_n,\fsr)$-duality, and irreducible modules from $(\GL_{n+1},\fsr)$-duality.

\begin{corollary}\label{coro bran d f}   The following statements hold.
\begin{itemize}
\item[(1)] Under the action of $\GL(n,\bbc)\times \calh_r$, the space of $r$-tensors over $\bbc^{n+1}$ decomposes as
$$(\bbc^{n+1})^{\otimes r} \cong\bigoplus_{l=0}^r\bigoplus_{\lambda\in \text{Par}(l,n)}L^{\lambda}_n\otimes D^{\lambda}_l.$$

\item[(2)] (Branching duality formula) In the Schur-Weyl duality,  the irreducible pairs $(L^\mu_{n+1}, S^\mu_{r})$ ($\mu\in \text{Par}(r,n+1)$) for   $(\GL_{n+1},\fsr)$ and the irreducible pairs $(L^\lambda_n, S^\lambda_l$) ($\lambda\in \text{Par}(l,n)$) for   $(\GL_n,\fsl)$ with $l=1,\ldots,r$ satisfy   the following branching duality formula
\begin{align}\label{bran d formula}
\sum\limits_{\mu\in \text{Par}(r,n+1)}\delta_{\lambda\lesssim\mu}\dim S^{\mu}_r={r\choose l}\dim S^{\lambda}_l
\end{align}
for any $\lambda\in \text{Par}(l,n)$ ($l\leq r$).
\end{itemize}
\end{corollary}

\begin{proof}
(1) Note that all $c\in\Gm$ act on $\uvtr_l$ by a scalar for any $l$. Keep Lemma \ref{ddha decomp l} in mind. So the decomposition follows from Theorem \ref{res swd} and the classical double commutant theorem (see \cite[\S4.16]{GW}).

(2) It follows from (1) that as a $\GL(n,\bbc)$-module
$$(\bbc^{n+1})^{\otimes r}\cong \bigoplus_{l=0}^r\bigoplus_{\lambda\in \text{Par}(l,n)}(L^{\lambda}_n)^{\bigoplus({r\choose l}\dim S^{\lambda}_l)}=
\sum\limits_{l=0}^r\sum\limits_{\lambda\in \text{Par}(l,n)}\big({r\choose l}\dim S^{\lambda}_l\big)L^{\lambda}_n.$$
On the other hand, according to the classical Schur-Weyl duality, $(\bbc^{n+1})^{\otimes r}$ has the following decomposition, as a $\GL(n+1,\bbc)$-module:
\begin{align*}
(\bbc^{n+1})^{\otimes r}&\cong\bigoplus_{\mu\in \text{Par}(r,n+1)}(L^{\mu}_{n+1})^{\oplus\dim S^{\mu}_r}\cr
&=\sum\limits_{\mu\in \text{Par}(r,n+1)}\sum\limits_{\lambda\in \text{Par}(l,n)}(\delta_{\lambda\lesssim\mu}\dim S^{\mu}_r) L^{\lambda}_n\cr
&=\sum\limits_{l=0}^r\sum\limits_{\lambda\in \text{Par}(l,n)}\sum\limits_{\mu\in \text{Par}(r,n+1)}(\delta_{\lambda\lesssim\mu}\dim S^{\mu}_r) L^{\lambda}_n
\end{align*}
Compare both decompositions. The equation (\ref{bran d formula}) follows.
\end{proof}

\begin{remark}
(1) When $l=r$, the formula (\ref{bran d formula}) becomes trivial.

(2) When $l=r-1$, the formula (\ref{bran d formula}) becomes
$$\sum\limits_{\mu\in \text{Par}(r,n+1)}\delta_{\lambda\lesssim\mu}\dim S^{\mu}_r=r\dim S^{\lambda}_{r-1}$$
for $\lambda\in \text{Par}(r-1,n)$. This formula can be regarded as a duality to the following classical branching rule (see \cite[Corollary 9.2.7]{GW})
\begin{align}\label{sym gp branch rule}
\text{Res}^{\fsr}_{\frak{S}_{r-1}}S^\mu_r\cong \bigoplus_{\lambda\in\text{Par}(r-1,n)}
\delta_{\lambda\lesssim\mu}S^{\lambda}_{r-1}.
\end{align}
\end{remark}

\section{Enhanced tensor invariants and parabolic Schur-Weyl duality}\label{sec 5}
Keep the notations as before. In this section, we give the answer to Question \ref{ques}, then establish the   parabolic Schur-Weyl dualities and give some applications.

\subsection{}
Recall the notations in \S\ref{enhanced space}. In particular,  $\uV$ has a basis $\{\eta_1,\ldots,\eta_n; \eta_{n+1}\}$ with $\eta_i\in V$ for $i=1,\ldots,n$; $\eta_{n+1}=\eta$. Then $\uvtr$ has a basis $\{\eta_{\bi}=\eta_{i_1}\otimes \eta_{i_2}\otimes\cdots\otimes\eta_{i_r}\mid \bi=(i_1,i_2,\ldots,i_n)\in \caln^r\}$.
Let us recall  that $\uvtr$ naturally becomes a representation space of $\cale(n,r)$ (and of $S(n+1,r)$ more generally) by definition as in (\ref{Ups rep}).

  We first have the following basic  observation.

\begin{lemma}\label{invar subspace} Denote by $\calp(n,r)_{l,0}$ the subspace of $\calp(n,r)$ spanned by $\xi_{\bi_l(n+1)^{r-l}, (n+1)^r}$ with $\bi_l$ ranging over $\un^l$,  and by  $\calp(n,r)_{\geq0,0}$ the direct sum of all $\calp(n,r)_{l,0}$ with $l=0,1,\ldots,r$. Then the image of $\calp(n,r)_{\geq0,0}$ is exactly $(\uvtr)^{\fsr}$.
\end{lemma}

\begin{proof} Recall that for any $l=0,1,\ldots,r$, $\uvtr_l$ has a basis consisting of $\eta_\bi\in \uvtr_{l}$ with $\bi$ ranging over $\caln^r_l$. Note that $\bi\sim \bi_{l(n+1)^{r-l}}$ with $\bi_l\in \un^l$. By definition,  $\xi_{\bi_l(n+1)^{r-l}, (n+1)^r} (\eta^{\otimes r})$ is a fundamental invariant in $(\uvtr_l)^{\fsr}$ arising from the $\fsr$-orbit of $\eta_\bi$ (see (\ref{Ups rep})) . Correspondingly, the image of $\calp(n,r)_{l,0}$ coincides with $(\uvtr_l)^{\fsr}$. When $l$ ranges over $\{0,1,\ldots,r\}$, the lemma follows.
\end{proof}

As a result, we have the following crucial lemma.

\begin{lemma}\label{lemma enhanced}
 For any $\phi\in \End_{\calp(n,r)}(\uvtr)$, $\phi$ must lie in $\bbc\Psi(\fsr)+D(n,r)_{\ann}$ where $D(n,r)_\ann:=\{\phi\in D(n,r)\mid \phi((\uvtr)^{\fsr})=0\}$.
\end{lemma}
\begin{proof} Note that $\End_{\calp(n,r)}(\uvtr)\subset \End_{\scrl(n,r)}(\uvtr)$. From  Theorem \ref{res swd}(2), $\End_{\calp(n,r)}(\uvtr)\subset D(n,r)$. In the following, we will show that $\End_{\calp(n,r)}(\uvtr)\cap D(n,r)\subset\bbc\Psi(\fsr)+D(n,r)_{\ann}$. By Lemma \ref{sum lem}(4),
$$D(n,r)=\bbc\Psi(\fsr)x_e\oplus ( \bigoplus_{l=1}^r\bigoplus_{\overset{(J,I)\in\mathscr{S}^2_l}{ i=1,\ldots,d(n,l)}}\bbc c_{J,I} x^I_{\sigma_{l,i}}).$$
Thanks to Proposition \ref{enhanced swty}, we only need to show that
\begin{align}\label{equ: inters 0 for ed}
\End_{\calp(n,r)}(\uvtr)\cap ( \bigoplus_{l=1}^r\bigoplus_{\overset{(J,I)\in\mathscr{S}^2_l}{ i=1,\ldots,d(n,l)}}\bbc c_{J,I} x^I_{\sigma_{l, i}})\subset D(n,r)_\ann.
\end{align}
Actually, for any given nonzero $\phi=\sum_{l=t}^s \phi_l$ with $0<t\leq s\leq r$, and $\phi_l\in D(n,r)_l$, $l=t,t+1,\ldots,s$, such that $\phi\in \End_{\calp(n,r)}(\uvtr)$ and both $\phi_s,\phi_t$ are not zero.
By the assumption, for any $\xi\in \calp(n,r)_{\geq0,0}$,  we have $\xi\circ \phi=\phi\circ \xi$. Then on one side $\xi\circ \phi$ vanishes because $\phi=\sum_{l=t}^s \phi_l$ with $t>0$. Hence, $\phi\circ\xi$ must be zero for all $\xi\in \calp(n,r)_{\geq0,0}$. Thanks to Lemma \ref{invar subspace}, the image of $\calp(n,r)_{\geq0,0}$ is exactly $(\uvtr)^{\fsr}$. This means that $\phi$ annihilates the space  $(\uvtr)^{\fsr}$. 

Summing up, (\ref{equ: inters 0 for ed}) is proved. The lemma follows.
\end{proof}

Generally, we have the following theorem.

\begin{theorem} (Parabolic Schur-Weyl duality) \label{enh sch thm} Keep the notations as above, in particular $\uG=\underline{\GL(V)}$. Set  $D(n,r)^V=\{\phi\in D(n,r)\mid e^v\circ\phi\circ e^{-v}=\phi, \forall v\in V\}$. Then the following statements hold.
\begin{itemize}
\item[(1)]  $\End_{\bbc\Phi(\uG\rtimes\Gm)}(\uvtr)=D(n,r)^V$.
\item[(2)]  The above $D(n,r)^V$ can be described as
$$\bbc\Psi(\fsr)\subset D(n,r)^V\subset \bbc\Psi(\fsr)+D(n,r)_\ann.$$
 \end{itemize}
\end{theorem}

\begin{proof} (1) Keep in mind $\calp(n,r)=\bbc\Phi(\uG\rtimes\Gm)$, and $\uG=G\times_\nu V$. Then the statement follows from Theorem \ref{res swd}.

(2) The first inclusion follows from Lemma \ref{enhanced swty}(1). The second one is due to Lemma \ref{lemma enhanced}.
\end{proof}

\subsection{Example} Let us demonstrate the above theorem by an example. 
Keep the notations as before, and let $I=\{1,2\}$, $I_i=\{i\}$ for $i=1,2$,  and $\sigma\in\mathfrak{S}_2$ interchanging $1$ and $2$. Then we have the following demonstration of $D(n,r)^V$.
\begin{itemize}
\item[(1)] If $r=1$ or $2$, and $n\geq r$,  then $D(n,r)^V=\bbc\Psi(\fsr)$.
\item[(2)] If $r=2$ and $n=1$, then $D(n,r)^V=\bbc\Psi(\fsr)\oplus\bbc(x_e^I-x_{\sigma}^I)$.
\end{itemize}

{\sl{Proof of the demonstration.}}
Recall that $\uV$ has a basis $\eta_1,\cdots, \eta_n; \eta_{n+1}:=\eta$, where $\eta_1,\cdots, \eta_n$ forms a basis of $V$.

(1) In the case $r=1$, it is obvious that $D(n,r)_{\ann}$=0. The assertion is obvious.

Now we suppose that $n\geq r=2$. Then $D(n, r)$ has the following  basis $$\{x_e^I,x_{\sigma}^I, \Psi(\sigma)x_e^{I_1}, \Psi(\sigma)x_e^{I_2},x_e^{I_1}, x_e^{I_2}, x_e^{\emptyset}\}.$$ Take any $\phi\in D(n,r)^V$, by the discussion as in Lemma \ref{lemma enhanced}, we can assume that
\begin{equation}\label{expression of phi}
\phi=a_1x_e^I+a_2x_{\sigma}^I+a_3\Psi(\sigma)x_e^{I_1}+a_4\Psi(\sigma)x_e^{I_2}+a_5x_e^{I_1}+a_6x_e^{I_2}+a_7x_e^{\emptyset},
\end{equation}
where $a_i\in\bbc$ for $1\leq i\leq 7$. By the assumption, $\phi\circ\Phi(e^{\eta_1})=\Phi(e^{\eta_1})\circ\phi$. In particular, on one hand,
\begin{align}\label{eq1 for lem5.4}
&\phi\circ\Phi(e^{\eta_1})(\eta\otimes\eta)\cr
=&\phi((\eta_1+\eta)\otimes(\eta_1+\eta))\cr
=&(a_1+a_2)\eta_1\otimes\eta_1+(a_3+a_6)\eta\otimes\eta_1+(a_4+a_5)\eta_1\otimes\eta+a_7\eta\otimes\eta,
\end{align}
and one the other hand,
\begin{align}\label{eq2 for lem5.4}
&\Phi(e^{\eta_1})\circ\phi(\eta\otimes\eta)\cr
=&\Phi(e^{\eta_1})(a_7\eta\otimes\eta)\cr
=&a_7(\eta_1+\eta)\otimes(\eta_1+\eta)\cr
=&a_7\eta_1\otimes\eta_1+a_7\eta\otimes\eta_1+a_7\eta_1\otimes\eta+a_7\eta\otimes\eta.
\end{align}
By comparing both sides of (\ref{eq1 for lem5.4}) and (\ref{eq2 for lem5.4}), we  have
\begin{equation}\label{eqq1}
a_1+a_2=a_3+a_6=a_4+a_5=a_7.
\end{equation}
Moreover, on one hand,
\begin{align}\label{eq3 for lem5.4}
&\phi\circ\Phi(e^{\eta_1})(\eta_2\otimes\eta)\cr
=&\phi(\eta_2\otimes (\eta_1+\eta))\cr
=&a_1\eta_2\otimes\eta_1+a_2\eta_1\otimes\eta_2+a_3\eta\otimes\eta_2+a_5\eta_2\otimes\eta.
\end{align}
On the other hand, \begin{align}\label{eq4 for lem5.4}
&\Phi(e^{\eta_1})\circ\phi(\eta_2\otimes\eta)\cr
=&\Phi(e^{\eta_1})(a_3\eta\otimes\eta_2+a_5\eta_2\otimes\eta)\cr
=&a_3\eta_1\otimes\eta_2+a_3\eta\otimes\eta_2+a_5\eta_2\otimes\eta_1+a_5\eta_2\otimes\eta.
\end{align}
By comparing both sides of (\ref{eq3 for lem5.4}) and (\ref{eq4 for lem5.4}), we  have
\begin{equation}\label{eqq2}
a_1=a_5, a_2=a_3.
\end{equation}
It follows from (\ref{eqq1}) and (\ref{eqq2}) that
$$a_1=a_5=a_6:=a, a_2=a_3=a_4=b, a_7=a+b.$$
Consequently,
$$\phi=a(x_e^I+x_e^{I_1}+x_e^{I_2}+x_e^{\emptyset})+b(x_{\sigma}^I+\Psi(\sigma)x_e^{I_1}+\Psi(\sigma)x_e^{I_2}+x_e^{\emptyset})
=a\Psi(e)+b\Psi(\sigma)\in\bbc\Psi(\fsr).$$
The assertion follows.

(2) Suppose $r=2$ and $n=1$. Take any $\phi\in D(n,r)^V$. As the arguments in (1), we can write $\phi$ as the form (\ref{expression of phi}).
By the assumption, $\phi\circ\Phi(e^{\eta_1})=\Phi(e^{\eta_1})\circ\phi$. In particular, on one hand,
\begin{align}\label{eq5 for lem5.4}
&\phi\circ\Phi(e^{\eta_1})(\eta_1\otimes\eta)\cr
=&\phi(\eta_1\otimes(\eta_1+\eta))\cr
=&(a_1+a_2)\eta_1\otimes\eta_1+a_3\eta\otimes\eta_1+a_5\eta_1\otimes\eta,
\end{align}
and one the other hand,
\begin{align}\label{eq6 for lem5.4}
&\Phi(e^{\eta_1})\circ\phi(\eta_1\otimes\eta)\cr
=&\Phi(e^{\eta_1})(a_3\eta\otimes\eta_1+a_5\eta_1\otimes\eta)\cr
=&(a_3+a_5)\eta_1\otimes\eta_1+a_3\eta\otimes\eta_1+a_5\eta_1\otimes\eta.
\end{align}
By comparing both sides of (\ref{eq5 for lem5.4}) and (\ref{eq6 for lem5.4}), we  have
\begin{equation}\label{eqq3}
a_1+a_2=a_3+a_5.
\end{equation}
Then it follow from (\ref{eqq1}) and (\ref{eqq3}) that
$$a_3=a_4:=c, a_5=a_6:=d, a_7=a_1+a_2=c+d.$$
Consequently,
\begin{align*}
\phi=&d(x_e^I+x_e^{I_1}+x_e^{I_2}+x_e^{\emptyset})+c(x_{\sigma}^I+\Psi(\sigma)x_e^{I_1}+\Psi(\sigma)x_e^{I_2}+x_e^{\emptyset})+(a_1-c)x_e^I+(a_2-d)x_{\sigma}^I)\cr
=&d\Psi(e)+c\Psi(\sigma)+(a_1-c)(x_e^I-x_{\sigma}^I)\in\bbc\Psi(\fsr)\oplus\bbc(x_e^I-x_{\sigma}^I).
\end{align*}
Moreover, it is a routine to check that $x_e^I-x_{\sigma}^I\in D(n,r)^V $, and the assertion follows.

\subsection{A conjecture} We  propose the following conjecture.
\begin{conjecture} When $n\geq r$, $D(n,r)^V$ coincides with $\bbc\Psi(\fsr)$.
\end{conjecture}

\subsection{Enhanced tensor invariants}
   Identify $\duvtr$ with $(\uvtr)^*$. Then  there is a natural $\uG\rtimes\Gm$-equivariant isomorphism of vector spaces
$$T:\uvtr\otimes {\duvtr}\rightarrow \End_\bbc(\uvtr).$$
With the parabolic Schur-Weyl duality, we describe the enhanced tensor invariants.

 Still let $\eta_1,\ldots,\eta_n$ be a basis of $V$, and let $\eta_{n+1}=\eta$. Then $\{\eta_i\mid i=1,\ldots, n+1\}$ constitute a basis of $\uV$. Let $\eta_1^*,\ldots,\eta_{n+1}^*$ be the dual basis of $\uV^*$. For a multi-index $\bi=(i_1,\ldots,i_r)\in \caln^r$. Still set $\eta_\bi=\eta_{i_1}\otimes\cdots\otimes \eta_{i_r}$, and $\eta^*_{\bi}=\eta^*_{i_1}\otimes\cdots\otimes \eta^*_{i_r}$. Then $\eta_\bi$, and $\eta^*_\bi$ form basis of $\uvtr$ and $\duvtr$ respectively, when $\bi$ ranges over $\caln^r$. Furthermore, the set $\{\eta_\bi\otimes \eta^*_\bj\mid (\bi,\bj)\in \caln^r\times\caln^r\}$ is a basis of $\uvtr\otimes \duvtr$.

For $\sigma\in \fsr$, define a mixed tensor $C_\sigma$ by $$C_\sigma=\sum_{\bi\in\caln^r} \eta_{\sigma.\bi}\otimes \eta^*_\bi.$$

\begin{proposition} Let $\uG=\underline{\GL(V)}$. When $D(n,r)^V=\bbc\Psi(\fsr)$, then the space of $\uG\rtimes\Gm$-invariants in the mixed-tensor product $\uvtr\otimes \duvtr$ is generated by  these $C_\sigma$ with $\sigma\in \fsr$.
\end{proposition}

\begin{proof} By the definition of $T$ mentioned before, for $\sigma\in\fsr$ we have $$T(C_\sigma)\eta_\bj=
\sum_{\bi\in\caln^r}\eta^*_\bi(\eta_\bj)\eta_{\sigma.\bi}=\eta_{\sigma.\bj}.$$
This means that $T(C_\sigma)=\Psi(\sigma)$. By the assumption,
the $\uG\rtimes\Gm$-invariants in $\uvtr\otimes \duvtr$ is isomorphic to $\sum_{\sigma\in\fsr}\bbc\Psi(\sigma)$ via $T$. Hence, the proposition is proved.
\end{proof}

\section{Representations of the parabolic Schur algebra}
Keep the notations as before. In this section, we proceed to study representations of the parabolic Schur algebra $\calp(n,r)$. In the whole section, the term ``a module of $\calp(n,r)$"  always means a  right module of $\calp(n,r)$.

\subsection{Irreducible modules and PIMs} Recall $\calp(n,r)$ has a sequence of right ideals of $\calp(n,r)$:
$$\calp(n,r)=\calp(n,r)_0\supset \calp(n,r)_1\supset\calp(n,r)_2\supset\cdots\supset\calp(n,r)_r\supset\calp(n,r)_{r+1}.$$
 with quotients $\calp(n,r)_i\slash \calp(n,r)_{i+1}\cong S(n,i)\ltimes\mathfrak{a}_i$ (Lemma \ref{enh sch stru}(3)).
 Those right ideals satisfy $\calp(n,r)_i\calp(n,r)=\calp(n,r)_i$. Naturally, $S(n,r)$ naturally becomes a $\calp(n,r)$-module.

 Consider the primitive idempotent decomposition of the identity element $\id$ in $\calp(n,r)$
\begin{align}\label{ide decomp}
\id=\sum_{\bs\in\bbn^n_r}\xi_{\tpis,\tpis}.
\end{align}
We have
\begin{align}\label{reali of IPM}
\calp(n,r)=\sum_{\bs\in\bbn^n_r}\xi_{\tpis,\tpis}\calp(n,r).
\end{align}
Then  $P_{\tpis}:=\xi_{\tpis,\tpis}\calp(n,r)$ is an indecomposable projective module (PIM\footnote{short for Principal Indecomposable Modules.} for short) of $\calp(n,r)$. Furthermore, for $0\leq l\leq r$ we set
$$\calp(n,r)_{[l]}:=\bbc\mbox{-span}\{\xi_{\tilde\pi_\bs,\bj}\in E\mid |\bs|=l\}.$$
Then by (S1)-(S3) again,
\begin{align}\label{reali of the lth}
\calp(n,r)_{[l]}=\sum_{\overset{\bs\in\bbn^n_r}{|\bs|=l}}P_{\tpis},
\end{align}


\subsubsection{}\label{subsec gen irr} Let us diverge by recalling some general results on irreducible modules for the classical Schur algebras. First of all, the Schur algebra $S(m,r)$ over $\bbc$ is semisimple.
For any $\pi\in\um^r$, $\pi$ lies in some $\fsr$-orbit from $\um^r$, say $\lambda=(\lambda_1,\ldots,\lambda_m)\in \Lambda(m,r)$. The vector $\lambda$  can also be regarded as a unordered partition of $r$ into $m$ parts, where zero parts being allowed. So the identity element $\id_{S(m,r)}$ of $S(m,r)$ has the following primitive idempotent decomposition $\id_{S(m,r)}=\sum_{\lambda\in\Lambda(m,r)}\xi_{\lambda,\lambda}$.
 Consequently, the right ideal generated by $\xi_{\lambda,\lambda}$ forms an irreducible module over $S(m,r)$, denoted by $\tilde D_\lambda$.

Consider the Weyl group $\cw$ of $\GL_m$ which is exactly $\fsm$, and the action of $\cw$ on $\um^r$. The action is on the left, $w\pi=(w(\pi_1),\ldots,w(\pi_r))$ for $w\in\cw$ and $\pi=(\pi_1,\ldots,\pi_r)\in\um^r$.
This action commutates with that of $\fsr$. Consequently, it is well-defined that $\fsm$ acts on $\Lambda(m,r)$ as $w\lambda=(\lambda_{w^{-1}(1)},\ldots,\lambda_{w^{-1}(m)})$ for $w\in\cw$ and
$\lambda\in \Lambda(m,r)$. Each $\cw$-orbit of $\Lambda(m,r)$ contains only one dominant weight $\gamma=(\gamma_1,\ldots,\gamma_m)$ which satisfies $\gamma_1\geq \gamma_2\geq\cdots \geq\gamma_m$. Denote by $\Lambda^+(m,r)$ the set of all dominant weights. Then $\Lambda^+(m,r)$ exactly consists of all partition of $r$ admitting at most number  $m$ of  nonzero parts,  coinciding with $\text{Par}(r,m)$.

\begin{lemma} Suppose  $\nu$ is a dominant weight, and shares the same $\fsm$-orbit with $\lambda$. Then two $S(m,r)$-irreducible modules $\tilde D_\lambda$ and $\tilde D_\nu$ are isomorphic.
\end{lemma}

\begin{proof} Recall that  $S(m,r)=\Phi(\GL(m,\bbc))$ with $\GL(m,\bbc)=\GL(W)$ associated to a given basis $\eta_i,i=1,\ldots,m$ of $W\cong \bbc^m$.
Suppose $\gamma=w(\lambda)$ for $w\in\cw$. Note that $w$ gives rise to an automorphism of $\GL(W)$ by permutation via $w$ on the basis $\eta_i, i=1,\ldots,m$. Naturally, $w$ gives rise to an automorphism of $S(m,r)$ which interchanges  two $S(m,r)$-irreducible modules $\tilde D_\lambda$ and $\tilde D_\gamma$. Hence both of them are isomorphic, as $S(m,r)$-modules.
\end{proof}

Thus the isomorphism classes of irreducible  $S(m,r)$-modules  are parameterized by $\Lambda^+(m,r)$. For $\lambda\in \Lambda(m,r)$, the representative of the corresponding isomorphism classes of irreducibles is denoted by $\tilde{D}_\gamma$ where $\gamma=\cw(\lambda)\cap \Lambda^+(m,r)$. Denote by $\ell(\gamma)$  the number of different elements in the $\cw$-orbit of $\gamma$. Then there is a direct sum decomposition of irreducible modules for $S(m,r)$:
\begin{align}\label{Schur decomp}
S(m,r)\cong \sum_{\gamma\in \Lambda^+(m,r)}\tilde{D}^{\oplus \ell(\gamma)}_\gamma.
\end{align}

\subsubsection{}\label{Structure of Simples and PIMs} Let us turn back to the parabolic Schur algebra $\calp(n,r)$.
For a given $\bs\in\bbn^n_r$ with $|\bs|=l\leq r$, $\tpis=(1^{s_1}\ldots n^{s_n}(n+1)^{r-l})$.
 By (S1)-(S3), $P_{\tpis}=\xi_{\tpis,\tpis}\calp(n,r)_l$ and $\xi_{\tpis,\tpis}\calp(n,r)_{l+1}=0$.
 So $P_{\tpis}$ can be decomposed into a direct sum of two subspaces
 \begin{align}\label{subspace P'}
 P'_{\tpis}=\bbc\mbox{-span}\{\xi_{{\tpis},\bp_l(n+1)^{r-l}}\mid \bp_l\in \un^l \}
 \end{align}
 and
 \begin{align}\label{subm maxQ}
 Q_{\tpis}=\bbc\mbox{-span}\{\xi_{{\tpis},\bq_l(n+1)^{r-l}}\mid \bq_l\in \caln^l, \text{rk}_{\un}(\bq_l)<l\}.
 \end{align}
   Furthermore $Q_{\tpis}=P_{\tpis}\calp(n,r)_{l+1}$, and it becomes a $\calp(n,r)$-submodule of $P_{\tpis}$. Moreover, $Q_{\tpis}$ is the unique maximal submodule of   $P_{\tpis}$, and the quotient $P_{\tpis}\slash Q_{\tpis}\cong \tilde D_{\tpis}$ as $S(n,l)$-modules. By the arguments above, the irreducible $S(n,l)$-module $\tilde D_{\tpis}$ is naturally an $\calp(n,r)$-module. Still denote by $\cw$ the Weyl group of $\GL_n$.  By the same arguments as in \S\ref{subsec gen irr}, $w\in \cw$ gives rise to an automorphism of $\calp(n,r)$. As $\calp(n,r)$-modules, $P_{\tpis}\cong P_\gamma$, and $\tilde D_{\tpis}\cong D_{\gamma}$ for  some $\gamma\in \Lambda^+(n,l)$. Moreover,
 $$\bigoplus\limits_{\stackrel{l=0}{l\neq |\bs|}}^r\calp(n,r)_{[l]}\subseteq\Ann_{\calp(n,r)}(D_\gamma)\subsetneqq \calp(n,r).$$
This implies that $D_{\gamma}\not\cong D_{\gamma^{\prime}}$ and $P_{\gamma}\not\cong P_{\gamma^{\prime}}$ if $\gamma\in\Lambda^+(n,l)$ and $\gamma^{\prime}\in\Lambda^+(n,l^{\prime})$ with $l\neq l^{\prime}$. Furthermore, by the classical theory of Schur algebras, $D_{\upsilon}\not\cong D_{\upsilon^{\prime}}$ and $P_{\upsilon}\not\cong P_{\upsilon^{\prime}}$ if $\upsilon, \upsilon^{\prime}\in\Lambda^+(n,l)$ and $\upsilon\neq\upsilon^{\prime}$.
  Summing up, we have the following classification results of irreducible   $\calp(n,r)$-modules and  of indecomposable projective $\calp(n,r)$-modules.

\begin{theorem}\label{thm:irr IPM csf} The set $\Lambda:=\{\gamma\in \Lambda^+(n,l)\mid l=0,\cdots, r\}$ parameterize both of the isomorphism classes of irreducible $\calp(n,r)$-modules and  of indecomposable projective $\calp(n,r)$-modules. The corresponding irreducible ({\sl{resp.}} indecomposable projective) modules are $D_{\gamma}$ ({\sl{resp.}} $P_{\gamma}$) with $\gamma\in \Lambda$.
\end{theorem}

\subsection{Cartan invariants and block degeneracy for $\calp(n,r)$} By the above arguments, the $P_\gamma$'s with $\gamma\in \bigcup_{l=0}^r \Lambda^+(n,l)$ form all PIM's for  $\calp(n,r)$, up to isomorphisms. In this subsection, we will precisely determine the Cartan invariants  for the parabolic Schur algebra $\calp(n,r)$. For that, we need the following preliminary result which describes the decomposition factors appearing in $Q_{\tpis}$ defined in (\ref{subm maxQ}) and their multiplicity.

\begin{lemma}\label{Structure of Q}
Keep the notations as before. Then for any given $\bs\in\bbn^n_r$ with $1\leq |\bs|=l\leq r$, the following equality holds in the Grothendieck group for the $\calp(n,r)$-module category.
\begin{align}\label{eq: Q decomp in Groth gp}
[Q_{\tpis}]
=\sum\limits_{i=0}^{l-1}\sum\limits_{\gamma\in\Lambda^+(n,i)}\ell(\gamma)[D_{\gamma}]
=\sum\limits_{i=0}^{l-1}[S(n,i)],
\end{align}
where $\ell(\gamma)$ is defined as before,  the number of different elements in the $\cw$-orbit of $\gamma$.
\end{lemma}

\begin{proof} The second equation of (\ref{eq: Q decomp in Groth gp}) is derived from (\ref{Schur decomp}). We only need to prove the first one.

Recall
$$ Q_{\tpis}=\bbc\mbox{-span}\{\xi_{{\tpis},\bq_l(n+1)^{r-l}}\mid \bq_l\in \caln^l, \text{rk}_{\un}(\bq_l)<l\}.$$
Set
$$ Q_{\tpis}^{(i)}=\bbc\mbox{-span}\{\xi_{{\tpis},\bq_l(n+1)^{r-l}}\mid \bq_l\in \caln^l,  \text{rk}_{\un}(\bq_l)\leq l-i\}\,\,\text{for}\,\, 1\leq j\leq l.$$
By (S1)-(S3), each $Q_{\tpis}^{(i)}$ is a $\calp(n,r)$-submodule of $Q_{\tpis}$, and we have the following decreasing sequence of $\calp(n,r)$-modules:

$$Q_{\tpis}= Q_{\tpis}^{(1)}\supset Q_{\tpis}^{(2)}\supset\cdots \supset Q_{\tpis}^{(l)}\supset 0.$$
Furthermore,
$$\bigoplus\limits_{\stackrel{j=0}{j\neq l-i}}^r\calp(n,r)_{[j]}$$
acts trivially on $Q_{\tpis}^{(i)}/Q_{\tpis}^{(i+1)}$ and
$Q_{\tpis}^{(i)}/Q_{\tpis}^{(i+1)}\cong S(n, l-i)$ as $S(n, l-i)$-modules. Consequently, $Q_{\tpis}^{(i)}/Q_{\tpis}^{(i+1)}\cong S(n, l-i)$ as $\calp(n, l-i)$-modules because $S(n,l-i)$ is a $\calp(n,r)$-module. Furthermore, by (\ref{Schur decomp}) we have
$$[Q_{\tpis}^{(i)}/Q_{\tpis}^{(i+1)}]=\sum\limits_{\gamma\in\Lambda^+(n,l-i)}\ell(\gamma)[D_{\gamma}].$$
Consequently,
\begin{align*}
 [Q_{\tpis}]&=\sum\limits_{i=1}^{l}[Q_{\tpis}^{(i)}/Q_{\tpis}^{(i+1)}]\cr
 &=\sum\limits_{i=1}^{l}\sum\limits_{\gamma\in\Lambda^+(n,l-i)}\ell(\gamma)[D_{\gamma}]\cr
 &=\sum\limits_{i=0}^{l-1}\sum\limits_{\gamma\in\Lambda^+(n,i)}\ell(\gamma)[D_{\gamma}].
\end{align*}
The proof is completed.
\end{proof}

As usual, we denote by $(P_\gamma: D_{\gamma^{\prime}})$ the multiplicity of the irreducible module $D_{\gamma^{\prime}}$ appearing in composition series of the indecomposable projective module $P_\gamma$ for
$\gamma,\gamma^{\prime}\in \bigcup_{l=0}^r \Lambda^+(n,l)$. We are now in the position to determine the Cartan invariants  for the parabolic Schur algebra $\calp(n,r)$.

\begin{theorem} \label{thm: Cartan inv}
Let $a_{\gamma,\, \gamma^{\prime}}:=(P_\gamma: D_{\gamma^{\prime}})$ for any $\gamma,\gamma^{\prime}\in \Lambda^+:=\bigcup_{l=0}^r \Lambda^+(n,l)$. Then
\begin{equation}\label{comp. number}
a_{\gamma,\ \gamma^{\prime}}=\begin{cases}
1, &\text{if}\,\, \gamma^{\prime}=\gamma;\cr
\ell(\gamma^{\prime}), &\text{if}\,\,|\gamma^{\prime}|<|\gamma|;\cr
0, &\text{otherwise}.\cr
\end{cases}
\end{equation}
Consequently, the Cartan matrix $(a_{\gamma,\gamma^{\prime}})_{\mathbbm{n}\times\mathbbm{n}}$ is an invertible and upper triangular one with diagonal entries being 1, where $\mathbbm{n}=\#\Lambda^+$.
\end{theorem}

\begin{proof}
If $\gamma\in\Lambda^+(n, 0)$, i.e., $\gamma=\textbf{0}=(0,\cdots, 0)$, then $P_{\textbf{0}}=D_{\textbf{0}}$ is irreducible, and the assertion (\ref{comp. number}) holds for $\gamma=\textbf{0}$ and any $\gamma^{\prime}\in\Lambda$. In the following, we assume $\gamma\in\Lambda^+(n, l)$ with $1\leq l\leq r$. By the discussion in \S\ref{Structure of Simples and PIMs},  $P_{\gamma}$ has a unique maximal submodule  $Q_{\gamma}$ with $P_{\gamma}/Q_{\gamma}\cong D_{\gamma}$. Hence, it follows from Lemma \ref{Structure of Q} that
$$ [P_{\gamma}]=[P_{\gamma}/Q_{\gamma}]+[Q_{\gamma}]=[D_{\gamma}]+\sum\limits_{i=0}^{l-1}\sum\limits_{\upsilon\in\Lambda^+(n,i)}\ell(\gamma)[D_{\upsilon}].$$
This implies the assertion (\ref{comp. number}). The other assertions are obvious. The proof is completed.
\end{proof}

As a direct consequence of Theorem \ref{thm: Cartan inv}, we have the following result on block structure of the parabolic Schur algebra.

\begin{corollary} The parabolic Schur algebra $\calp(n,r)$ has only one block.
\end{corollary}

The block structure is closely related to the center for a finite-dimensional algebra (see \cite[\S1.8]{Ben}). As to the latter, we have the following observation.

\begin{proposition}\label{center} When $D(n,r)^V=\bbc\Psi(\fsr)$, the center of $\calp(n,r)$ is one-dimensional.
\end{proposition}

     \begin{proof} Denote by $\sfC$ the center of $\calp(n,r)$. Note that $\calp(n,r)\subset\End_{\bbc}(\uvtr)$. We have
 $$\sfC\subset \mathsf{Cent}_{\calp(n,r)}(\End_{\bbc}(\uvtr)).$$
 Hence $\sfC= \End_{\calp(n,r)}(\uvtr)\cap \calp(n,r)$.  When $D(n,r)^V=\bbc\Psi(\fsr)$,   $\sfC= \bbc\Psi(\fsr)\cap \calp(n,r)$ by Theorem \ref{enh sch thm}.   Recall $\calp(n,r)$ has a basis $\{\xi_{\tpis,\bj}\mid(\tpis,\bj)\in E\}$ (Theorem \ref{enh sch stru}).
 For any given nonzero $\phi\in\sfC$, $\phi$ can be written in two forms
 \begin{align}\label{cent exp1}
 \phi=\sum_{(\tilde{\pi}_\bs,\bj)\in E} a_{\tilde\pi_\bs\bj}\xi_{\tilde\pi_\bs,\bj}
 \end{align}
 and
 \begin{align}\label{cent exp2}
 \phi=\sum_{\delta\in\fsr}b_\delta\Psi(\delta).
\end{align}
We refine (\ref{cent exp2}) as $\phi=\sum_{\delta\in S}b_\delta\Psi(\delta)$ where $S\subset\fsr$ and all $b_\delta\ne0$ for $\delta\in S$. Then we finally write (\ref{cent exp2}) as
\begin{align*}
 \phi&=\sum_{\delta\in S}b_\delta\sum_{\bs\in\bbn^n_r}\xi_{\tpis,\delta.\tpis}\cr
 &=\sum_{\bs\in\bbn^n_r}\sum_{\delta\in S}b_\delta\xi_{\tpis,\delta.\tpis}.
\end{align*}
If $S\ne\{\id\}$, then we can take $\delta$ $(\ne\id) \in S$. In this case, there certainly exists $\bs\in \bbn^n_r$ such that $(\tpis,\delta(\tpis))\notin E$. This contradicts with (\ref{cent exp1}).

So it is only possible that $S$ coincides with $\{\id\}$. Consequently, $\phi\in\bbc\id$.
Hence $\sfC=\bbc\id$. The proof is completed.
\end{proof}

\section*{Acknowledgement} We are grateful to Chenliang Xue for his pointing out some wrong arguments and statements in an old version of the manuscript.

\end{document}